\documentclass[11pt]{article}

\usepackage{amsmath}
\usepackage{amssymb}
\usepackage{amsfonts}
\usepackage{amsthm}
\usepackage{array}
\usepackage{bm}
\usepackage{enumitem}

\usepackage[small,nohug,heads=vee]{diagrams}
\diagramstyle[labelstyle=\scriptstyle]
\usepackage{diagrams}

\newcolumntype{C}[1]{>{\centering\hspace{0pt}}p{#1}}
\usepackage[pdftex]{graphicx}
\usepackage[margin=1in]{geometry}

\newcommand{\R}{\mathbb{R}}
\newcommand{\V}{\mathsf{V}}

\newcommand{\K}{\mathsf{K}}
\newcommand{\Ls}{\mathsf{L}}
\newcommand{\sph}{\ensuremath{\textrm{Spin}^h \left( 4 \right)}}
\newcommand{\spsev}{\ensuremath{\text{Spin}(7)}}

\newtheorem{thm}{Theorem}[section]
\newtheorem{prop}[thm]{Proposition}
\newtheorem{lem}[thm]{Lemma}
\newtheorem{cor}[thm]{Corollary}

\theoremstyle{definition}
\newtheorem{defn}[thm]{Definition}

\theoremstyle{definition}
\newtheorem*{rmk}{Remark}

\usepackage[nottoc]{tocbibind}
\numberwithin{equation}{section}

\usepackage[hidelinks]{hyperref}

\title{The Mean Curvature of First-Order Submanifolds \\ in Exceptional Geometries with Torsion}
\author{Gavin Ball and Jesse Madnick}
\date{September 2019}

\newcommand{\Addresses}
{{  \bigskip
  \textsc{Universit\'{e} du Qu\'{e}bec \`{a} Montr\'{e}al} \par\nopagebreak
  \textsc{D\'{e}partement de math\'{e}matiques}\par\nopagebreak
  \textsc{Case postale 8888, succursale centre-ville, Montr\'{e}al (Qu\'{e}bec), H3C 3P8, Canada}\par\nopagebreak
  \textit{E-mail address}: \texttt{gavin.cf.ball@gmail.com} \\

 \bigskip
  \textsc{McMaster University} \par\nopagebreak
  \textsc{Department of Mathematics \& Statistics}\par\nopagebreak
  \textsc{Hamilton, ON, Canada, L8S 4K1}\par\nopagebreak
 \textit{E-mail address}: \texttt{madnickj@mcmaster.ca}

}}

\begin{document}

\maketitle

\begin{abstract}
We derive formulas for the mean curvature of associative $3$-folds, coassociative $4$-folds, and Cayley $4$-folds in the general case where the ambient space has intrinsic torsion.  Consequently, we are able to characterize those $\text{G}_2$-structures (resp., $\text{Spin}(7)$-structures) for which every associative $3$-fold (resp. coassociative $4$-fold, Cayley $4$-fold) is a minimal submanifold. \\
\indent In the process, we obtain new obstructions to the local existence of coassociative $4$-folds in $\text{G}_2$-structures with torsion.
\end{abstract}

\tableofcontents

 \pagebreak
 
\section{Introduction}

%

\indent \indent  In their fundamental work on calibrations, Harvey and Lawson \cite{harvey1982calibrated} defined four new classes of calibrated submanifolds in Riemannian manifolds with special holonomy, summarized in the following table:
\begin{center}
  \begin{tabular}{ | l | l |}
    \hline
    Submanifold & Ambient Manifold \\ \hline \hline
    Special Lagrangian $n$-fold & Riemannian $2n$-manifold $(M^{2n}, g)$ with $\text{Hol}(g) \leq \text{SU}(n)$ \\ \hline
    Associative $3$-fold & Riemannian $7$-manifold $(M^7, g)$ with $\text{Hol}(g) \leq \text{G}_2$ \\ \hline
    Coassociative $4$-fold & Riemannian $7$-manifold $(M^7, g)$ with $\text{Hol}(g) \leq \text{G}_2$ \\ \hline
    Cayley $4$-fold & Riemannian $8$-manifold $(M^8, g)$ with $\text{Hol}(g) \leq \text{Spin}(7)$  \\ \hline
  \end{tabular}
\end{center}

\noindent By virtue of being calibrated, each of these submanifolds satisfy a strong area-minimizing property.  In particular, they are stable minimal submanifolds.  Moreover, by an argument using the Cartan-K\"{a}hler Theorem, Harvey and Lawson \cite{harvey1982calibrated} were able to show that submanifolds of each class exist locally in abundance.

Riemannian manifolds with special holonomy groups function as the background spaces for supersymmetric theories of physics. In this setting, calibrated submanifolds are related to supersymmetric cycles \cite{BECKER96}. Special Lagrangian submanifolds lie at the foundation of the SYZ formulation of mirror symmetry \cite{SYZ96}, and calibrated submanifolds in manifolds with holonomy groups $\text{G}_2$ and $\spsev$ are expected to play a similar role in theories of mirror symmetry for such manifolds \cite{ACHARYA98,GYZ03}.

\indent In fact, each of the classes of submanifolds described above make sense in an even more general class of ambient spaces: namely, that of (Riemannian) manifolds $M$ equipped with a $G$-structure, for $G = \text{SU}(n)$ or $\text{G}_2$ or $\text{Spin}(7)$ as appropriate.
\begin{center}
  \begin{tabular}{ | l | l |}
    \hline
    Submanifold & Ambient Manifold \\ \hline \hline
    Special Lagrangian $n$-fold & $2n$-manifold $M^{2n}$ with an $\text{SU}(n)$-structure \\ \hline
    Associative $3$-fold & $7$-manifold $M^7$ with a $\text{G}_2$-structure \\ \hline
    Coassociative $4$-fold & $7$-manifold $M^7$ with a $\text{G}_2$-structure \\ \hline
    Cayley $4$-fold & $8$-manifold $M^8$ with a $\text{Spin}(7)$-structure \\ \hline
  \end{tabular}
\end{center}

\noindent However, in this generalized setting, such submanifolds need not be minimal.  This raises the following: \\


\noindent \textbf{Minimality Problem:} Let $M$ be a manifold.  Characterize those $G$-structures (for $G = \text{SU}(n)$, $\text{G}_2$, $\text{Spin}(7)$) on $M$ for which every submanifold in $M$ of a given class (special Lagrangian, associative, etc.) is a minimal submanifold of $M$. \\


\indent We will completely solve the Minimality Problem in the contexts of associative $3$-folds, and coassociative $4$-folds, and Cayley $4$-folds by deriving simple formulas for their mean curvature.  The case of special Lagrangian $3$-folds is addressed in our preprint \cite{BalMadsLag19}.  \\ 
\indent Perhaps more fundamentally, in our generalized context the relevant submanifolds need not exist at all, even locally.  This raises the natural: \\

\noindent \textbf{Local Existence Problem:} Let $M$ be a manifold.  Characterize those $G$-structures (with $G$ as above) on $M$ for which submanifolds of a given class (special Lagrangian, etc.) exist locally at every point of $M$. \\


\indent In this work, we make progress towards the resolution of the Local Existence Problem in the setting of coassociatives.  More precisely, we obtain an explicit obstruction to the local existence of coassociative $4$-folds.  Analogous obstructions to the local existence of special Lagrangian $3$-folds are obtained in \cite{BalMadsLag19}.

\subsection{Results on Associative $3$-folds and Coassociative $4$-folds}

\indent \indent Let $(M^7, \varphi)$ be a $7$-manifold with a $\text{G}_2$-structure $\varphi \in \Omega^3(M)$.  The first-order local invariants of $\varphi$ are completely encoded in four differential forms, called the \textit{torsion forms} of the $\text{G}_2$-structure, denoted
$$(\tau_0, \tau_1, \tau_2, \tau_3) \in \Omega^0 \oplus \Omega^1 \oplus \Omega^2 \oplus \Omega^3.$$
These are defined by the equations
\begin{align*}
d\varphi & = \tau_0 \ast\!\varphi + 3\tau_1 \wedge \varphi + \ast\tau_3  \\
d\ast\! \varphi & = \ \ \ \ \ \ \ \ \ \ 4\tau_1 \wedge \ast \varphi + \tau_2 \wedge \varphi.
\end{align*}
together with the algebraic relations $\ast(\varphi \wedge \tau_2) = -\tau_2$ and $\tau_3 \wedge \varphi = 0$ and $\tau_3 \wedge \ast\varphi = 0$. \\
\indent In order to study associative $3$-folds and coassociative $4$-folds in $M$, we will break the torsion forms into $\text{SO}(4)$-irreducible pieces with respect to a certain splitting of $TM$.  Indeed, in \S\ref{ssect:G2RefTors}, we will decompose $\tau_0, \tau_1, \tau_2, \tau_3$ into $\text{SO}(4)$-irreducible components, writing
\begin{align*}
\tau_0 & = \tau_0 & \tau_2 & = (\tau_2)_{\mathsf{A}} + (\tau_2)_{1,3} + (\tau_2)_{2,0} \\
\tau_1 & = (\tau_1)_{\mathsf{A}} + (\tau_1)_{\mathsf{C}} & \tau_3 & = (\tau_3)_{0,0} + (\tau_3)_{0,4} + (\tau_3)_{2,2} + (\tau_3)_{1,3} + (\tau_3)_{\mathsf{C}}
\end{align*}
We will refer to the individual pieces
$$\tau_0, \ \ \ \ (\tau_1)_{\mathsf{A}}, (\tau_1)_{\mathsf{C}}, \ \ \ \ (\tau_2)_{\mathsf{A}}, (\tau_2)_{1,3}, (\tau_2)_{2,0}, \ \ \ \ (\tau_3)_{0,0}, (\tau_3)_{0,4}, (\tau_3)_{2,2}, (\tau_3)_{1,3}, (\tau_3)_{\mathsf{C}}$$
as \textit{refined torsion forms} (with respect to a certain splitting of $TM$).   The mean curvature of associatives and coassociatives can then be expressed purely in terms of the refined torsion.  In the sequel, we let $\sharp \colon T^*M \to TM$ denote the usual musical (index-raising) isomorphism. \\

\noindent \textbf{Theorem \ref{thm:MCAssoc} (Mean Curvature of Associatives):} The mean curvature vector $H$ of an associative $3$-fold in $M$ is given by
$$H = \textstyle -3 [(\tau_1)_{\mathsf{C}}]^\sharp - \frac{\sqrt{3}}{2} \left[(\tau_3)_{\mathsf{C}}\right]^\ddag$$
where $\ddag$ is a particular isometric isomorphism defined (\ref{eq:G2-DDagIsom}). \\
\indent In particular, a $\text{G}_2$-structure on $M$ has the property that every associative $3$-fold in $M$ is minimal if and only if $\tau_1 = \tau_3 = 0$.  Equivalently, if and only if $d\varphi = \lambda \ast\! \varphi$ for some constant $\lambda \in \mathbb{R}$.  \\

\noindent \textbf{Theorem \ref{thm:MCCoass}  (Mean Curvature of Coassociatives):} The mean curvature vector $H$ of a coassociative $4$-fold in $M$ is given by
$$H = \textstyle -4 [(\tau_1)_{\mathsf{A}}]^\sharp + \frac{\sqrt{6}}{3} \left[(\tau_2)_{\mathsf{A}}\right]^\natural$$
where $\natural$ is a particular isometric isomorphism defined in (\ref{eq:G2-NatIsom}). \\
\indent In particular, a $\text{G}_2$-structure on $M$ has the property that every coassociative $4$-fold in $M$ is minimal if and only if $\tau_1 = \tau_2 = 0$.  Equivalently, if and only if $d\ast\!\varphi = 0$.  \\

\indent These formulas can be regarded as a submanifold analogue of the curvature formulas derived by Bryant \cite{Bryant:2006aa} for $7$-manifolds with $\text{G}_2$-structures.  In the process of proving Theorem \ref{thm:MCCoass}, we will observe an obstruction to the local existence of coassociative $4$-folds: \\

\noindent \textbf{Theorem \ref{thm:CoassObs} (Local Obstruction to Coassociatives):} If a coassociative $4$-fold $\Sigma$ exists in $M$, then the following relation holds at points of $\Sigma$:
$$\tau_0 = \textstyle -\frac{\sqrt{42}}{7}\,[(\tau_3)_{0,0}]^\dagger$$
where $\dagger$ is an isometric isomorphism defined in (\ref{eq:G2-DagIsom}). \\
\indent In particular, if $\tau_3 = 0$ and $\tau_0$ is non-vanishing, then $M$ admits no coassociative $4$-folds (even locally). \\

\noindent \textbf{Corollary \ref{cor:CoassObs}:} Fix $x \in M$. If every coassociative $4$-plane in $T_xM$ is tangent to a coassociative $4$-fold, then $\tau_0|_x = 0$ and $\tau_3|_x = 0$. \\

\noindent Note that Theorem \ref{thm:CoassObs} generalizes the well-known fact that nearly-parallel $\text{G}_2$-structures (viz., those with $\tau_1 = \tau_2 = \tau_3 = 0$ and $\tau_0$ non-vanishing) admit no coassociative $4$-folds.

\subsection{Results on Cayley $4$-folds}

\indent \indent Let $(M^8, \Phi)$ be an $8$-manifold with a $\text{Spin}(7)$-structure $\Phi \in \Omega^4(M)$.  The first-order local invariants of $\Phi$ are completely encoded in two differential forms, $\tau_1 \in \Omega^1(M)$ and $\tau_3 \in \Omega^3(M)$, called the \textit{torsion forms} of the $\text{Spin}(7)$-structure.  They are defined by equation
$$d\Phi = \tau_1 \wedge \Phi + \ast \tau_3.$$
To study Cayley $4$-folds in $M$, we will break the torsion forms into $\text{Spin}^h(4)$-irreducible pieces with respect to a certain splitting of $TM$, where here $\text{Spin}^h(4) = (\text{SU}(2) \times \text{SU}(2) \times \text{SU}(2))/\mathbb{Z}_2$ is the stabilizer of a Cayley $4$-plane.  Indeed, in $\S$\ref{ssect:CayRefTors}, we will decompose $\tau_1$ and $\tau_3$ into irreducible pieces, writing
\begin{align*}
\tau_1 & = (\tau_1)_{\mathsf{K}} + (\tau_1)_{\mathsf{L}} \\
\tau_3 & = (\tau_3)_{\mathsf{K}} + (\tau_3)_{\mathsf{L}} + (\tau_3)_{0,3,1} + (\tau_3)_{2,1,1} + (\tau_3)_{1,3,0} + (\tau_3)_{1,1,2}
\end{align*}
and refer to the individual pieces
$$(\tau_1)_{\mathsf{K}},  (\tau_1)_{\mathsf{L}}, \ \ \ \ \ (\tau_3)_{\mathsf{K}}, (\tau_3)_{\mathsf{L}}, (\tau_3)_{0,3,1},  (\tau_3)_{2,1,1}, (\tau_3)_{1,3,0}$$
as \textit{refined torsion forms} (with respect to a certain splitting of $TM$).  As with associative and coassociative submanifolds, the mean curvature of Cayley submanifolds can be expressed purely in terms of the refined torsion.  As before, we let $\sharp \colon T^*M \to TM$ denote the musical isomorphism. \\

\noindent \textbf{Theorem \ref{thm:CayleyMC} (Mean Curvature of Cayleys):} The mean curvature vector $H$ of a Cayley $4$-fold in $M$ is given by
\begin{align*}
H = - \left[ \left( \tau_1 \right)_\Ls \right]^\sharp - \tfrac{\sqrt{42}}{7} \left[ \left( \tau_3 \right)_\Ls \right]^\dagger
\end{align*}
where $\dagger$ is a particular isometric isomorphism defined in Definition \ref{defn:dagger}. \\
\indent In particular, a $\spsev$-structure on $M$ has the property that every Cayley $4$-fold in $M$ is minimal if and only if $d\Phi = 0$.


\subsection{Organization}

\indent \indent In $\S$\ref{sect:G2}, we study associative $3$-folds and coassociative $4$-folds in $7$-manifolds with $\text{G}_2$-structures.  We use $\S$\ref{ssect:G2Prelims} to recall basic aspects of $\text{G}_2$ geometry and set notation.  In $\S$\ref{ssect:SO4reps}, we will decompose various $\text{G}_2$-modules (viz., $\Lambda^k(\mathbb{R}^7)$ and $\text{Sym}^2(\mathbb{R}^7)$) into $\text{SO}(4)$-irreducible pieces, where we think of $\text{SO}(4)$ as the stabilizer of an associative (or coassociative) plane.  These $\text{SO}(4)$-submodules are central to the geometry of associative and coassociative submanifolds, and we will take care to provide explicit descriptions of these submodules as much as possible. \\
\indent In $\S$\ref{ssect:G2RefTors}, we will define the refined torsion forms by way of the $\text{SO}(4)$-irreducible decompositions obtained in $\S$\ref{ssect:SO4reps}.  Once these refined torsion forms are defined, we will express them in terms of a local $\text{SO}(4)$-frame in order to perform calculations. \\
\indent Finally, in $\S$\ref{ssect:MCAssoc} we will apply this machinery to the study of associative $3$-folds, proving Theorem  \ref{thm:MCAssoc}.  Similarly, in \S\ref{ssect:MCCoassoc}, we turn to coassociative $4$-folds, proving Theorem \ref{thm:CoassObs}, Corollary \ref{cor:CoassObs}, and Theorem \ref{thm:MCCoass}. \\
\indent The structure of \S\ref{sect:Cayley} is completely analogous.  In brief, we use $\S$\ref{ssect:Spin7prelim} to recall the basic aspects of $\text{Spin}(7)$ geometry and set notation.  In $\S$\ref{ssect:sphreps}, we decompose various $\text{Spin}(7)$-modules into $\text{Spin}^h(4)$-representations, where we continue to write $\text{Spin}^h(4) = (\text{SU}(2) \times \text{SU}(2) \times \text{SU}(2))/\mathbb{Z}_2$, regarded as the stabilizer of a Cayley $4$-plane.  In $\S$\ref{ssect:CayRefTors}, we will define the corresponding refined torsion forms by way of $\text{Spin}^h(4)$-representation theory, and then express them in terms of a local $\text{Spin}^h(4)$-frame.  Finally, in $\S$\ref{ssect:MCCay}, we will study Cayley $4$-folds and prove Theorem \ref{thm:CayleyMC}. \\


\noindent \textbf{Acknowledgements:} This work has benefited from conversations with Robert Bryant, Jason Lotay, Thomas Madsen, and Alberto Raffero.  The second author would also like to thank McKenzie Wang for his guidance and encouragement. The first author thanks the Simons Collaboration on Special Holonomy in Geometry, Analysis and Physics for support during the period in which this article was written.  \\

\indent After this work was completed, it was pointed out to us that the mean curvature formula for associative $3$-folds, Theorem \ref{thm:MCAssoc}, was derived earlier by Paul Reynolds in his 2011 PhD thesis \cite{reynolds11}.  In that work, Reynolds also derived an expression for the mean curvature of Cayley $4$-folds, but his expression is not completely split into irreducible pieces.


\section{Associative $3$-Folds and Coassociative $4$-Folds in $\text{G}_2$-Structures} \label{sect:G2}


\indent \indent Our goal in this section is to derive formulas (Theorems \ref{thm:MCAssoc} and \ref{thm:MCCoass}) for the mean curvature of associative $3$-folds and coassociative $4$-folds in $7$-manifolds equipped with a $\text{G}_2$-structure.  We will also derive an obstruction (Theorem \ref{thm:CoassObs}) to the local existence of coassociative $4$-folds.

These formulas and obstructions will be phrased in terms of \textit{refined torsion forms}, which we will define in \S\ref{ssect:G2RefTors}.  These refined forms are essentially the $\text{SO}(4)$-irreducible pieces of the usual torsion forms $\tau_0, \tau_1, \tau_2, \tau_3$ of a $\text{G}_2$-structure.  As such, we will devote \S\ref{ssect:SO4reps} to the relevant $\text{SO}(4)$-representation theory needed to decompose $\tau_0, \tau_1, \tau_2, \tau_3$.

\subsection{Preliminaries}\label{ssect:G2Prelims}

\indent \indent In this subsection, we define both the ambient spaces (in \S\ref{sssect:G2Structsn7Mflds}) and submanifolds (in \S\ref{sssect:AssocAndCoassoc}) of interest.  We also use this subsection to fix notation and clarify conventions.

\subsubsection{$\text{G}_2$-Structures on Vector Spaces}\label{sssect:G2structs}

\indent \indent Let $V = \mathbb{R}^7$ equipped with the standard inner product $\langle \cdot, \cdot \rangle$,  norm $\Vert \cdot \Vert$, and an orientation.  Let $\{e_1, \ldots, e_7\}$ denote the standard (orthonormal) basis of $V$, and let $\{e^1, \ldots, e^7\}$ denote the corresponding dual basis of $V^*$.  The \textit{associative $3$-form} is the alternating $3$-form $\phi_0 \in \Lambda^3(V^*)$ defined by
\begin{equation*}
\phi_0 = e^{123} + e^1 \wedge (e^{45} + e^{67}) + e^2 \wedge (e^{46} - e^{57}) + e^3 \wedge (-e^{47} - e^{56})
\end{equation*}
The \textit{coassociative $4$-form} is the alternating $4$-form $\ast\phi_0 \in \Lambda^4(V^*)$ given by the Hodge dual $\ast$ of $\phi_0$.  Explicitly:
\begin{equation*}
\ast\phi_0 = e^{4567} + e^{23} \wedge (e^{45} + e^{67}) + e^{13} \wedge (- e^{46} + e^{57} ) + e^{12} \wedge (- e^{47} - e^{56}).
\end{equation*}
For calculations, it will be convenient to express $\phi_0$ and $\ast \phi_0$ in the form
\begin{align*}
\phi_0 & = \textstyle \frac{1}{6}\epsilon_{ijk}\,e^{ijk} & \ast\phi_0 & = \textstyle \frac{1}{24}\epsilon_{ijk\ell}\,e^{ijk\ell}
\end{align*}
where the constants $\epsilon_{ijk}, \epsilon_{ijk\ell} \in \{-1,0,1\}$ are defined by this formula.  For example, $\epsilon_{123} = \epsilon_{145} = 1$ and $\epsilon_{347} = \epsilon_{356} = -1$.  Identities involving the $\epsilon$-symbols are given in \cite{Bryant:2006aa}.
\begin{rmk}
	The associative and coassociative forms admit simple descriptions via the algebra of the octonions $\mathbb{O}$. \\
	\indent Equip $\mathbb{O} \simeq \mathbb{R}^8$ with the standard (euclidean) inner product and split $\mathbb{O} = \text{Re}(\mathbb{O}) \oplus \text{Im}(\mathbb{O}) \simeq \mathbb{R} \oplus \mathbb{R}^7$, where $\text{Re}(\mathbb{O}) := \text{span}_{\mathbb{R}}(1)$ is the real line and $\text{Im}(\mathbb{O}) := \text{Re}(\mathbb{O})^\perp$ is its orthogonal complement.  Under the identification $V \simeq\text{Im}(\mathbb{O})$, the associative and coassociative forms are given by
	\begin{align*}
	\phi_0(x,y,z) & = \langle x, y \times z \rangle \\
	\ast\phi_0(x,y,z,w) & = \textstyle \frac{1}{2}\langle x, [y,z,w]\rangle,
	\end{align*}
	for $x,y,z \in V$, where $y \times z := \text{Im}(\overline{z}y) = \frac{1}{2}(\overline{z}y - \overline{y}z)$ is the octonionic cross product, and $[y,z,w] := (yz)w - y(zw)$ is the \textit{associator}, measuring the failure of associativity of multiplication in $\text{Im}(\mathbb{O})$.  See \cite{harvey1982calibrated} for a proof.
\end{rmk}

\indent Consider the $\text{GL}(V)$-action on $\Lambda^3(V^*)$ given by pullback: $A \cdot \gamma := A^*\gamma$ for $A \in \text{GL}(V)$ and $\gamma \in \Lambda^3(V^*)$.  It is a classical result of Schouten (see [Bryant 87] for a proof) that the stabilizer of $\phi_0 \in \Lambda^3(V^*)$ is the compact Lie group $\text{G}_2$, i.e.:
$$\text{G}_2 \cong \{A \in \text{GL}(V) \colon A^*\phi_0 = \phi_0\}.$$
We let $\Lambda^3_+(V^*)$ denote the orbit of $\phi_0 \in \Lambda^3(V^*)$ under this $\text{G}_2$-action, i.e.:
$$\Lambda^3_+(V^*) := \{A^*\phi_0 \colon A \in \text{GL}(V)\} \cong \frac{\text{GL}(V)}{\text{G}_2}.$$
In \cite{Bryant:2006aa}, it is noted that $\Lambda^3_+(V^*) \subset \Lambda^3(V^*) \simeq \mathbb{R}^{35}$ is an open subset with two connected components, each diffeomorphic to $\mathbb{RP}^7 \times \mathbb{R}^{28}$. \\

\indent The isomorphism $\text{G}_2 \cong \{A \in \text{GL}(V) \colon A^*\phi_0 = \phi_0\}$ lets us regard $\text{G}_2$ as a subgroup of $\text{GL}(V)$, which in turn lets us view $V \simeq \mathbb{R}^7$ as a faithful $\text{G}_2$-representation.  It can be shown (see [Bryant 87]) that this $\text{G}_2$-representation is irreducible. \\
\indent However, the induced $\text{G}_2$-representations on $\Lambda^k(V^*)$ for $2 \leq k \leq 5$ are not irreducible.  Indeed, $\Lambda^2(V^*)$ decomposes into irreducible $\text{G}_2$-modules as
$$\Lambda^2(V^*) = \Lambda^2_7 \oplus \Lambda^2_{14},$$
where
\begin{align*}
\Lambda^2_7 & = \{\beta \in \Lambda^2(V^*) \colon \ast\!(\phi_0 \wedge \beta) = 2\beta\} \\
\Lambda^2_{14} & = \{\beta \in \Lambda^2(V^*) \colon \ast\!(\phi_0 \wedge \beta) = -\beta\}
\end{align*}
Similarly, $\Lambda^3(V^*)$ decomposes into irreducible $\text{G}_2$-modules as
$$\Lambda^3(V^*) = \Lambda^3_1 \oplus \Lambda^3_7 \oplus \Lambda^3_{27}$$
where
\begin{align*}
\Lambda^3_1 & = \mathbb{R}\phi_0 \\
\Lambda^3_7 & = \{\ast(\alpha \wedge \phi_0) \colon \alpha \in \Lambda^1\} \\
\Lambda^3_{27} & = \{\gamma \in \Lambda^3 \colon \gamma \wedge \phi_0 = 0 \text{ and } \gamma \wedge \ast\phi_0 = 0\}.
\end{align*}
In each case, $\Lambda^k_\ell$ is an irreducible $\text{G}_2$-module of dimension $\ell$. Via the Hodge star $\ast \colon \Lambda^k(V^*) \to \Lambda^{7-k}(V^*)$, one can obtain similar decompositions of $\Lambda^4(V^*)$ and $\Lambda^5(V^*)$. \\

\indent In the sequel, we will always equip $\Lambda^k(V^*)$ with the usual inner product, also denoted $\langle \cdot, \cdot \rangle$, given by declaring
\begin{equation}
\{e^I \colon I \text{ increasing multi-index}\} \label{eq:G2-InnerProd}
\end{equation}
to be an orthonormal basis.  We let $\Vert \cdot \Vert$ denote the corresponding norm.  \\

\indent For our calculations in \S\ref{ssect:SO4reps}, we will need the $\text{G}_2$-equivariant map $\mathsf{i}$, defined on decomposable elements of $\text{Sym}^2_0(V^*)$ as follows:
\begin{align}
\mathsf{i} \colon \text{Sym}^2_0(V^*) & \to \Lambda^3(V^*) \label{eq:G2-imap} \\
\mathsf{i}(\alpha \circ \beta) & = \alpha \wedge \ast(\beta \wedge \ast \phi_0) + \beta \wedge \ast(\alpha \wedge \ast\phi_0). \notag
\end{align}
It is shown in \cite{Bryant:2006aa} that the image of $\mathsf{i}$ is $\Lambda^3_{27}$, so that the map with restricted image $\mathsf{i} \colon \text{Sym}^2_0(V^*) \to \Lambda^3_{27}$ is an isomorphism of $\text{G}_2$-modules.  It is also remarked that with respect to the orthonormal basis $\{e^1, \ldots, e^7\}$ of $V^*$, one has
$$\mathsf{i}(h_{ij}\,e^i \circ e^j) = \epsilon_{ik\ell}\,h_{ij}\,e^{jk\ell}.$$
To invert $\mathsf{i}$, one can use the map
\begin{align*}
\mathsf{j} \colon \Lambda^3_{27}(V^*) & \to \text{Sym}^2_0(V^*) \\
\mathsf{j}(\gamma)(v,w) & = \ast( \iota_v\phi_0 \wedge \iota_w\phi_0 \wedge \gamma)
\end{align*}
which satisfies $\mathsf{j} \circ \mathsf{i} = 8\,\text{Id}_{\text{Sym}^2_0(V^*)}$. \\

\indent Finally, we remark that from the associative $3$-form $\phi_0$, one can recover the inner product $\langle \cdot, \cdot \rangle$ and volume form $\text{vol} = e^{1 \cdots 7}$ via
\begin{subequations} \label{eq:G2-VolRel}
\begin{align} 
\langle X,Y \rangle\,\text{vol} & = \textstyle \frac{1}{6}\,(\iota_X\phi_0) \wedge (\iota_Y \phi_0) \wedge \phi_0  \\
\text{vol} & = \phi_0 \wedge \ast \phi_0. 
\end{align}
\end{subequations}
From these identities, one can show that, in fact, $\text{G}_2$ preserves both $\langle \cdot, \cdot \rangle$ and the orientation on $V$, so we may regard $\text{G}_2 \leq \text{SO}(V, \langle \cdot, \cdot \rangle) \simeq \text{SO}(7)$.

\subsubsection{$\text{G}_2$-Structures on $7$-Manifolds}\label{sssect:G2Structsn7Mflds}

\begin{defn}
	Let $M$ be an oriented $7$-manifold.  A \textit{$\text{G}_2$-structure} on $M$ is a differential $3$-form $\varphi \in \Omega^3(M)$ such that $\varphi|_x \in \Lambda^3_+(T_x^*M)$ at each $x \in M$.  That is, at each $x \in M$, there exists a coframe $u \colon T_xM \to \mathbb{R}^7$ for which $\varphi|_x = u^*(\phi_0)$.
\end{defn}

Intuitively, a $\text{G}_2$-structure is a smooth identification of each tangent space $T_xM$ with $\text{Im}(\mathbb{O})$ in such a way that $\varphi|_x$ and $\phi_0$ are aligned: $(T_xM, \varphi|_x) \simeq (\text{Im}(\mathbb{O}), \phi_0)$.  We remark that a $7$-manifold $M$ admits a $\text{G}_2$-structure if and only if it is orientable and spin: see \cite{Bryant:2006aa} for a proof. \\
\indent Every $\text{G}_2$-structure $\varphi$ on $M$ induces a Riemannian metric $g_\varphi$ and an orientation form $\text{vol}_\varphi$ on $M$ via the formulas (\ref{eq:G2-VolRel}), reflecting the inclusion $\text{G}_2 \leq \text{SO}(7)$.  We caution, however, that the association $\varphi \mapsto g_\varphi$ is not injective: different $\text{G}_2$-structures may induce the same Riemannian metric.  For a discussion of this point, see \cite{Bryant:2006aa}.

The first-order local invariants of a $\text{G}_2$-structure are completely encoded in a certain $\text{G}_2$-equivariant function
$$T \colon F_{\text{G}_2} \to \Lambda^0 \oplus \Lambda^1 \oplus \Lambda^2_{14} \oplus \Lambda^3_{27} \simeq \mathbb{R}^{49}$$
called the \textit{intrinsic torsion function}, defined on the total space of the $\text{G}_2$-frame bundle $F_{\text{G}_2} \to M$ over $M$.  We think of $T$ as describing the $1$-jet of the $\text{G}_2$-structure.

The intrinsic torsion function is somewhat technical to define --- the interested reader can find more information in \cite{MR1062197} and \cite{MR1004008} --- but several equivalent reformulations are available.  Most conveniently for our purposes: the intrinsic torsion function of a $\text{G}_2$-structure is equivalent to the data of the $4$-form $d\varphi$ and the $5$-form $d\!\ast\!\varphi$.  In \cite{Bryant:2006aa}, the exterior derivatives of $\varphi$ and $\ast\varphi$ are shown to take the form
\begin{subequations} \label{eq:G2-TorsionEqns}
\begin{align}
d\varphi & = \tau_0 \ast\!\varphi + 3\tau_1 \wedge \varphi + \ast\tau_3  \\
d\ast\! \varphi & = \ \ \ \ \ \ \ \ \ \ 4\tau_1 \wedge \ast \varphi + \tau_2 \wedge \varphi.
\end{align}
\end{subequations}
where
$$(\tau_0, \tau_1, \tau_2, \tau_3) \in \Gamma(\Lambda^0(T^*M) \oplus \Lambda^1(T^*M) \oplus \Lambda^2_{14}(T^*M) \oplus \Lambda^3_{27}(T^*M))$$
We refer to $\tau_0, \tau_1, \tau_2, \tau_3$ as the \textit{torsion forms} of the $\text{G}_2$-structure. \\ 
\indent Following standard conventions, we let $W_1, W_7, W_{14}, W_{27}$ denote the vector bundles $\Lambda^0(T^*M)$, $\Lambda^1(T^*M)$, $\Lambda^2_{14}(T^*M)$, $\Lambda^3_{27}(T^*M)$, respectively.  Consider the set $\mathcal{S}$ consisting of the $2^4 = 16$ vector bundles
$$\mathcal{S} = \left\{0, \ W_i, \ W_i \oplus W_j, \ W_i \oplus W_j \oplus W_k, \ W_1 \oplus W_7 \oplus W_{14} \oplus W_{27} \colon i,j,k \in \{1,7,14,27\}\right\}\!.$$

\begin{defn}
	Let $E \in \mathcal{S}$ be a vector bundle on the list above. We say that a $\text{G}_2$-structure belongs to the \textit{torsion class $E$} if and only if the torsion forms of the $\text{G}_2$-structure $(\tau_0, \tau_1, \tau_2, \tau_3) \in \Gamma(W_1 \oplus W_7 \oplus W_{14} \oplus W_{27})$ is valued in $E \subset W_1 \oplus W_7 \oplus W_{14} \oplus W_{27}$.
\end{defn}
\noindent For example, a $\text{G}_2$-structure belongs to the torsion class $W_7 \oplus W_{27}$ if and only if $\tau_0 = \tau_2 = 0$.

\subsubsection{Associative $3$-Folds and Coassociative $4$-Folds}\label{sssect:AssocAndCoassoc}


\indent \indent Let $(M^7, \varphi)$ be a $7$-manifold with a $\text{G}_2$-structure, and consider a tangent space $(T_xM, \varphi|_x) \simeq (V, \phi_0)$.  The vector space $(V, \phi_0)$ possesses two distinguished classes of subspaces --- associative $3$-planes and coassociative $4$-planes (to be defined shortly) ---  first studied by Harvey and Lawson \cite{harvey1982calibrated} in their work on calibrations.  Indeed, they observed that $\phi_0$ and $\ast\phi_0$ enjoy the following remarkable property:

\begin{prop}[\cite{harvey1982calibrated}]\label{prop:G2Comass}
	The associative $3$-form $\phi_0$ and coassociative $4$-form $\ast\phi_0$ have co-mass one, meaning that
	\begin{align*}
	\left|\phi_0(x,y,z)\right| & \leq 1 & \left|\ast\phi_0(x,y,z,w)\right| & \leq 1
	\end{align*}
	for every orthonormal set $\{x,y,z,w\}$ in $V \simeq \mathbb{R}^7$.
\end{prop}

\indent In light of this proposition, it is natural to examine more closely those $3$-planes $A \in \text{Gr}_3(V)$ (respectively, $4$-planes $C \in \text{Gr}_4(V)$) for which $\left|\phi_0(A)\right| = 1$ (resp., $\left|\ast\phi_0(C)\right| = 1$).

\begin{prop}[\cite{harvey1982calibrated}] \label{prop:AssocDefn} 
	Let $A \in \text{Gr}_3(V)$ be a $3$-plane in $V$.  The following are equivalent:
	\begin{enumerate}[nosep]
		\item  If $\{u,v,w\}$ orthonormal basis of $A$, then $\phi_0(u,v,w) = \pm 1$.
		\item For all $u,v,w \in A$, we have $[u,v,w] = 0$.
		\item $A = \text{span}\{u,v, u \times v\}$ for some linearly independent set $\{u,v\}$.
	\end{enumerate}
	If any of these conditions hold, we say that $A$ is an \emph{associative $3$-plane}.
\end{prop}

\begin{prop}[\cite{harvey1982calibrated}]\label{prop:CoassocDefn} 
	 Let $C \in \text{Gr}_4(V)$ be a $4$-plane in $V$.   The following are equivalent:
	 \begin{enumerate}[nosep]
	 	\item If $\{x,y,z,w\}$ is an orthonormal basis of $C$, then $\ast\phi_0(x,y,z,w) = \pm 1$.
	 	\item $C^\perp$ is associative.
	 	\item $\phi_0|_C = 0$.
	 \end{enumerate}
	If any of these conditions hold, we say that $C$ is a \emph{coassociative $4$-plane}.
\end{prop}

\noindent Proofs of the above propositions can be found in \cite{harvey1982calibrated} and \cite{salamon2010notes}. \\

\indent The $\text{G}_2$-action on $V$ induces $\text{G}_2$-actions on the Grassmannians $\text{Gr}_k(V)$ of $k$-planes in $V$.  While these actions are transitive for $k = 1,2,5,6$, they are not transitive for $k = 3,4$: indeed, the (proper) subsets consisting of associative $3$-planes and coassociative $4$-planes are $\text{G}_2$-orbits.  The corresponding stabilizer, recorded in the following proposition, will play a crucial role in this work:
\begin{prop}[\cite{harvey1982calibrated}] \label{prop:G2ActStab} 
	The Lie group $\text{G}_2$ acts transitively on the subset of associative $3$-planes and on the subset of coassociative $4$-planes:
	\begin{align*}
	\{E \in \text{Gr}_3(V) \colon \left|\phi_0(E)\right| = 1\} & \subset \text{Gr}_3(V), \\
	\{E \in \text{Gr}_4(V) \colon \left|\ast\phi_0(E)\right| = 1\} & \subset \text{Gr}_4(V).
	\end{align*}
	In both cases, the stabilizer of the $\textup{G}_2$-action is isomorphic to $\textup{SO}(4)$.
\end{prop}

\indent We may finally define our primary objects of interest:
\begin{defn}
	Let $(M^7, \varphi)$ be a $7$-manifold equipped with a $\text{G}_2$-structure $\varphi$.  Identify each tangent space $(T_xM, \varphi|_x) \simeq (V, \phi_0)$. \\
	\indent An \textit{associative $3$-fold} in $M$ is a $3$-dimensional immersed submanifold $\Sigma \subset M$ for which each tangent space $T_x\Sigma \subset T_xM$ is an associative $3$-plane. \\
	\indent Similarly, a \textit{coassociative $4$-fold} in $M$ is a $4$-dimensional immersed submanifold $\Sigma \subset M$ for which each tangent space $T_x\Sigma \subset T_xM$ is a coassociative $4$-plane.
\end{defn}

\indent Note that if $d\varphi = 0$, then $\varphi$ is a calibration whose calibrated $3$-planes are the associative $3$-planes in $T_xM$.  Thus, in this case, an associative $3$-fold is a calibrated submanifold, and hence a minimal submanifold of $M$. \\
\indent Similarly, if $d\!\ast\! \varphi = 0$, then $\ast\varphi$ is a calibration whose calibrated $4$-planes are the coassociative $4$-planes in $T_xM$.  Thus, in this case, a coassociative $4$-fold is a calibrated submanifold, and hence a minimal submanifold of $M$.  The ``Minimality Problem" described in the introduction asks about the converses of these claims.


\subsection{Some $\text{SO}(4)$-Representation Theory}\label{ssect:SO4reps}

\indent \indent In this subsection, we describe the aspects of $\text{SO}(4)$-representation theory that are relevant to the study of associative $3$-folds and coassociative $4$-folds.  Particularly important for our purposes are $\text{SO}(4)$-irreducible decompositions of $\Lambda^1(\mathbb{R}^7)$, $\Lambda^2(\mathbb{R}^7)$, $\text{Sym}^2(\mathbb{R}^7)$, and $\Lambda^3(\mathbb{R}^7)$, each of which we will describe in turn. \\

\indent To begin, recall that the compact Lie group $\text{SO}(4)$ is double-covered by the simply-connected group $\text{SU}(2) \times \text{SU}(2).$  The complex irreducible representations of $\text{SU}(2) \times \text{SU}(2)$ are exactly the tensor products $\V_p \otimes \V_q$ of irreducible $\text{SU}(2)$-representations for each factor. The complex irreducible representations of $\text{SU}(2)$ are well known to be the spaces of homogeneous polynomials in two variables of fixed degree, $\V_p = \text{Sym}^p \left(\mathbb{C}\langle x,y\rangle \right)$. \\
\indent Let $\V^{\mathbb{C}}_{p,q}$ denote $\V_p \otimes \V_q$.  We think of $\V^{\mathbb{C}}_{p,q}$ as the space of homogeneous polynomials in $(x,y;w,z)$ of bidegree $(p,q)$.  When $p$ and $q$ have the same parity the representation $\V^{\mathbb{C}}_{p,q}$ descends to a representation of $\text{SO}(4)$, and each of these representations has a real structure induced by the map $(x,y,w,z) \mapsto (y,-x,z,-w).$ This yields a complete description of the real representations of $\text{SO}(4)$. \\
\indent We work with real representations, letting $\V_{p,q}$ denote the real representation underlying  $\V^{\mathbb{C}}_{p,q}$. In this language, the standard 4-dimensional representation of $\text{SO}(4)$ is $\V_{1,1}$, while the adjoint representation $\mathfrak{so}(4)$ is $\V_{2,0} \oplus \V_{0,2}$. The ordering of the subscripts is chosen so that the representation $\Lambda^2_+(\mathbb{R}^4) $ of $\text{SO}(4)$ on the self-dual 2-forms is $\mathsf{V}_{0,2}$. \\
\indent The Clebsch-Gordan formula applied to each $\text{SU}(2)$ representation gives the irreducible decomposition of a tensor product of $\text{SO}(4)$-modules:
\begin{align*}\label{eq:ClebschGSO4}
\V_{p_1,q_1} \otimes \V_{p_2,q_2} \cong \bigoplus_{i = 0}^{|p_1-p_2|} \bigoplus_{j=0}^{|q_1 - q_2|} \V_{p_1+p_2-2i, \,q_1+q_2-2j}.
\end{align*}

\subsubsection{$\text{SO}(4)$ as a subgroup of $\text{G}_2$}\label{sssect:SO4subG2}

\indent \indent In our calculations we shall need a concrete realization of $\text{SO}(4)$ as the stabilizer of an associative or coassociative plane. Let $\text{SO}(4)$ act on $V \cong \R^7$ via the identification $V \cong \V_{0,2} \oplus \V_{1,1},$ and let $(e_1, \ldots, e_7)$ be an orthonormal basis of $V$ such that:
\begin{itemize}[nosep]
\item $ \left\langle e_1, e_2, e_3 \right\rangle \cong \V_{0,2}$ and $\left\langle e_4, e_5, e_6, e_7 \right\rangle \cong \V_{1,1},$
\item The map 
\begin{align*}
e_1 \mapsto e_{45} + e_{67}, \ \ \ e_2 \mapsto e_{46} - e_{57}, \ \ \ e_3 \mapsto -e_{47}-e_{56} 
\end{align*}
is $\text{SO}(4)$-equivariant. \\
\end{itemize}
Then the 3-form
\begin{align*}
e^{123} + e^1 \wedge (e^{45} + e^{67}) + e^2 \wedge (e^{46} - e^{57}) + e^3 \wedge (-e^{47} - e^{56})
\end{align*}
is $\text{SO}(4)$-invariant, and thus the action of $\text{SO}(4)$ on $V$ gives an embedding $\text{SO}(4) \subset \text{G}_2$. The 3-plane $\left\langle e_1, e_2, e_3 \right\rangle$ is associative and preserved by the action of $\text{SO}(4)$, while the 4-plane $\left\langle e_4, e_5, e_6, e_7 \right\rangle$ is coassociative and preserved by the action of $\text{SO}(4).$

\subsubsection{Decomposition of $1$-Forms on $V^*$}\label{sssect:G21FormDecomp}

\indent \indent Let $V$ be as above.  The decomposition of $\Lambda^1(V^*)$ into irreducible $\text{SO}(4)$-modules is simply
$$\Lambda^1(V^*) = \mathsf{A} \oplus \mathsf{C},$$
where
\begin{align*}
\mathsf{A} & \cong \left\langle e^1, e^2, e^3 \right\rangle\!, &\mathsf{C} & \cong \left\langle e^4, e^5, e^6, e^7 \right\rangle\!.
\end{align*}
As abstract $\text{SO}(4)$-modules, we have isomorphisms $\mathsf{A} \cong \V_{0,2}$ and $\mathsf{C} \cong \V_{1,1}.$ \\

\noindent \textbf{Notation:} We let $\flat \colon V \to V^*$ via $X^\flat := \langle X, \cdot \rangle$ denote the usual musical (index-lowering) isomorphism induced by the inner product $\langle \cdot, \cdot \rangle$ on $V$, and let
\begin{equation}
\sharp \colon V^* \to V \label{eq:G2-Sharp}
\end{equation}
denote its inverse.  In the sequel, we let $\mathsf{A}^\sharp, \mathsf{C}^\sharp \subset V$ denote the images of $\mathsf{A}, \mathsf{C} \subset V^*$ under the $\sharp$ isomorphism.

\subsubsection{Decomposition of $2$-Forms on $V^*$}\label{sssect:G22FormDecomp}

\indent \indent We now seek to decompose $\Lambda^2(V^*)$ into $\text{SO}(4)$-irreducible submodules.  As noted in \S\ref{ssect:G2Prelims} above, $\Lambda^2(V^*)$ splits into $\text{G}_2$-irreducible submodules as
\begin{equation} \label{eq:G2-2forms-1}
\Lambda^2(V^*) \cong \Lambda^2_7 \oplus \Lambda^2_{14} 
\end{equation}
On the other hand, using $V^* = \mathsf{A} \oplus \mathsf{C}$, we may also decompose $\Lambda^2(V^*)$ as
\begin{equation} \label{eq:G2-2forms-2}
\Lambda^2(V^*) \cong \Lambda^2(\mathsf{A}) \oplus (\mathsf{A} \otimes \mathsf{C}) \oplus \Lambda^2_+(\mathsf{C}) \oplus \Lambda^2_-(\mathsf{C}).
\end{equation}
We will refine both decompositions (\ref{eq:G2-2forms-1}) and (\ref{eq:G2-2forms-2}) into $\text{SO}(4)$-submodules. \\
\indent To begin, note first that as $\text{SO}(4)$-modules, we have that $\Lambda^2(\mathsf{A}) \cong \Lambda^2_+(\mathsf{C}) \cong \mathsf{V}_{0,2}$ and $\Lambda^2_-(\mathsf{C}) \cong \mathsf{V}_{2,0}$ are irreducible.  Thus, it remains only to decompose $\Lambda^2_7$, $\Lambda^2_{14}$, and $\mathsf{A} \otimes \mathsf{C}$.
\begin{defn}
	We define
	\begin{align*}
	(\Lambda^2_7)_{\mathsf{A}} & := \Lambda^2_7 \cap (\Lambda^2(\mathsf{A}) \oplus \Lambda^2(\mathsf{C})) = \{\iota_X\phi_0 \colon X \in \mathsf{A}^\sharp\} \\
	(\Lambda^2_{7})_{\mathsf{C}} & :=  \Lambda^2_7 \cap (\mathsf{A} \otimes \mathsf{C}) = \{\iota_X\phi_0 \colon X \in \mathsf{C}^\sharp\} \\
	\\
	(\Lambda^2_{14})_{\mathsf{A}} & := \Lambda^2_{14} \cap (\Lambda^2(\mathsf{A}) \oplus \Lambda^2_+(\mathsf{C})) \\
	(\Lambda^2_{14})_{1,3} & := \Lambda^2_{14} \cap (\mathsf{A} \otimes \mathsf{C}) \\
	(\Lambda^2_{14})_{2,0} & := \Lambda^2_{14} \cap \Lambda^2_-(\mathsf{C}). 
	\end{align*}
\end{defn}
\noindent The reader can check that, in fact, $\Lambda^2_-(\mathsf{C}) \subset \Lambda^2_{14}$, so that $(\Lambda^2_{14})_{2,0} = \Lambda^2_{14} \cap \Lambda^2_-(\mathsf{C}) = \Lambda^2_-(\mathsf{C})$. \\

\indent Consider the $\text{SO}(4)$-module isomorphism
\begin{align*}
L \colon \mathsf{A} & \to (\Lambda^2_7)_{\mathsf{A}} \\
L(\alpha) & = \iota_{\alpha^\sharp}\phi_0 = \ast(\alpha \wedge \ast\phi_0).
\end{align*}
For $\beta \in \Lambda^2(V^*)$, write $\beta = \beta|_{\Lambda^2(\mathsf{A})} + \beta|_{\mathsf{A} \otimes \mathsf{C}} + \beta|_{\Lambda^2(\mathsf{C})}$, where $\beta|_E \in E$ for $E \in \{\Lambda^2(\mathsf{A}), \mathsf{A} \otimes \mathsf{C}, \Lambda^2(\mathsf{C})\}$.  Define $\text{SO}(4)$-equivariant maps
\begin{align*}
L_{\mathsf{A}} \colon \mathsf{A} & \to \Lambda^2(\mathsf{A}) & L_{\mathsf{C}} \colon  \mathsf{A} & \to \Lambda^2_+(\mathsf{C}) \\
L_{\mathsf{A}}(\alpha) & = L(\alpha)|_{\Lambda^2(\mathsf{A})} & L_{\mathsf{C}}(\alpha) & = L(\alpha)|_{\Lambda^2(\mathsf{C})}
\end{align*}
It is straightforward to check that $L_{\mathsf{A}}$ and $L_{\mathsf{C}}$ are well-defined $\text{SO}(4)$-module isomorphisms, and that $L_{\mathsf{A}} = \ast_{\mathsf{A}}$ coincides with the usual Hodge star operator on $\mathsf{A}$.  Finally, we define the map
\begin{align*}
W \colon \mathsf{A} & \to (\Lambda^2_{14})_{\mathsf{A}} \\
W(\alpha) & = 2L_{\mathsf{A}}(\alpha) - L_{\mathsf{C}}(\alpha)
\end{align*}
Again, the reader can check that $W$ is a well-defined $\text{SO}(4)$-module isomorphism.  We caution that the maps $L$, $L_{\mathsf{C}}$, and $W$ are not isometries. 
\begin{lem}\label{lem:G22FormDecomp} 
	The decompositions
	\begin{align}
	\Lambda^2_7 & = (\Lambda^2_7)_{\mathsf{A}} \oplus (\Lambda^2_7)_{\mathsf{C}} \label{eq:Lam27decomp} \\
	\Lambda^2_{14} & = (\Lambda^2_{14})_{\mathsf{A}} \oplus (\Lambda^2_{14})_{1,3} \oplus (\Lambda^2_{14})_{2,0} \label{eq:Lam214decomp} \\
	\mathsf{A} \otimes \mathsf{C} & = (\Lambda^2_7)_{\mathsf{C}} \oplus (\Lambda^2_{14})_{1,3} \label{eq:ACdecomp}
	\end{align}
	consist of $\text{SO}(4)$-irreducible submodules.
	
	Thus, the decomposition
	\begin{align*}
	\Lambda^2(V) & = \left[(\Lambda^2_7)_{\mathsf{A}} \oplus (\Lambda^2_7)_{\mathsf{C}}\right] \oplus \left[(\Lambda^2_{14})_{\mathsf{A}} \oplus (\Lambda^2_{14})_{1,3} \oplus (\Lambda^2_{14})_{2,0}\right]
	\end{align*}
	is $\text{SO}(4)$-irreducible and refines (\ref{eq:G2-2forms-1}), while
	\begin{align*}
	\Lambda^2(V) & = \Lambda^2(\mathsf{A}) \oplus \left[(\Lambda^2_7)_{\mathsf{C}} \oplus (\Lambda^2_{14})_{1,3}\right] \oplus \Lambda^2_+(\mathsf{C}) \oplus \Lambda^2_-(\mathsf{C})
	\end{align*}
	is $\text{SO}(4)$-irreducible and refines (\ref{eq:G2-2forms-2}).
\end{lem}
\begin{proof}
	The decomposition (\ref{eq:Lam27decomp}) follows from the isomorphism $V \to \Lambda^2_7, \: X \mapsto \iota_X \!\left( \varphi_0 \right)$ and the irreducible decomposition $V \cong \mathsf{A} \oplus \mathsf{C}$. \\
	\indent By a dimension count, the $\text{SO}(4)$-invariant subspace $\Lambda^2_{14}$ of $\Lambda^2 \left( V^* \right)$ must be isomorphic to the $\text{SO}(4)$-module $\mathsf{V}_{0,2} \oplus \mathsf{V}_{2,0} \oplus \mathsf{V}_{1,3},$ while, by the Clebsch-Gordan formula, the space $\mathsf{A} \otimes \mathsf{C}$ is isomorphic to $\mathsf{V}_{1,3} \oplus \mathsf{V}_{1,1}.$ It follows from Schur's lemma that the space $(\Lambda^2_{14})_{1,3}$ is isomorphic to the $\text{SO}(4)$-module $\mathsf{V}_{1,3}.$ The space $(\Lambda^2_{14})_{\mathsf{A}}$ is the image of $\mathsf{A}$ under the isomorphism $W$ defined above, so it is an irreducible $\text{SO}(4)$-module, while the space $(\Lambda^2_{14})_{2,0}$ is isomorphic to $\Lambda^2_- ( \mathsf{C}) \cong \mathsf{V}_{2,0}$ so it is an irreducible $\text{SO}(4)$-module. Thus, the decomposition (\ref{eq:Lam214decomp}) consists of irreducible $\text{SO}(4)$-modules. \\
	\indent To see that (\ref{eq:ACdecomp}) is an irreducible decomposition, note that we have already shown that both $ (\Lambda^2_7)_{\mathsf{C}}$ and $(\Lambda^2_{14})_{1,3}$ are irreducible $\text{SO}(4)$-modules.
\end{proof}

\begin{defn}
	We define the map
	\begin{align}
	\natural \colon (\Lambda^2_{14})_{\mathsf{A}} & \to \mathsf{A}^\sharp \label{eq:G2-NatIsom} \\
	\beta & \mapsto \beta^\natural = \sqrt{6}\,(W^{-1}(\beta))^\sharp \notag
	\end{align}
	The map $\natural$ is an $\text{SO}(4)$-module isomorphism, and (because of the factor of $\sqrt{6}$) an isometry with respect to the inner product (\ref{eq:G2-InnerProd}) on $\Lambda^2(V^*)$.
\end{defn}

\subsubsection{Decomposition of the Quadratic Forms on $V^*$}\label{ssect:G2QuadFormDecomp}

\indent \indent Before turning to the decomposition of $\Lambda^3(V^*)$, we take a moment to decompose $\text{Sym}^2(V^*)$ into $\text{SO}(4)$-irreducible pieces.  To this end, we first use $V^* = \mathsf{A} \oplus \mathsf{C}$ to split
\begin{align*}
\text{Sym}^2(V^*) & \cong \mathbb{R}\text{Id}_{\mathsf{A}} \oplus \text{Sym}^2_0(\mathsf{A}) \oplus (\mathsf{A} \otimes \mathsf{C}) \oplus \mathbb{R}\text{Id}_{\mathsf{C}} \oplus \text{Sym}^2_0(\mathsf{C}).
\end{align*}
Each of these summands is $\text{SO}(4)$-irreducible, with the exception of $\mathsf{A} \otimes \mathsf{C}$, which splits into irreducible summands as
\begin{equation*}
\mathsf{A} \otimes \mathsf{C} \cong (\mathsf{A} \otimes \mathsf{C})_{1,3} \oplus (\mathsf{A} \otimes \mathsf{C})_{\mathsf{C}}
\end{equation*}
where $(\mathsf{A} \otimes \mathsf{C})_{1,3}$ and $(\mathsf{A} \otimes \mathsf{C})_{\mathsf{C}}$ are submodules isomorphic to $\mathsf{V}_{1,3}$ and $\mathsf{C}$, respectively. \\
\indent Note that we are employing a slight abuse of notation.  That is, in \S\ref{sssect:SO4subG2} we used $\mathsf{A} \otimes \mathsf{C}$ to denote a submodule of $\Lambda^2(V^*)$, whereas here in \S\ref{sssect:G21FormDecomp} we are using the same symbol $\mathsf{A} \otimes \mathsf{C}$ to denote a submodule of $\text{Sym}^2(V^*)$.  Abstractly, these two $\text{SO}(4)$-modules are isomorphic, as are their irreducible summands.  By Schur's Lemma, there is a one-dimensional family of $\text{SO}(4)$-module isomorphisms $(\Lambda^2_7)_{\mathsf{C}} \cong (\mathsf{A} \otimes \mathsf{C})_{\mathsf{C}}$ and $(\Lambda^2_{14})_{1,3} \cong (\mathsf{A} \otimes \mathsf{C})_{1,3}$. \\
\indent For computations, we will make use of the particular $\text{SO}(4)$-module isomorphism
\begin{equation}
\mathsf{s} \colon (\Lambda^2_7)_{\mathsf{C}} \to (\mathsf{A} \otimes \mathsf{C})_{\mathsf{C}} \label{eq:smap} 
\end{equation}
defined as follows.  In terms of a basis $\{e^1, \ldots, e^6\}$ of $V^*$ with $\mathsf{A} = \text{span}(e^1, e^2, e^3)$ and $\mathsf{C} = \text{span}(e^4, e^5, e^6)$, the map $\mathsf{s}$ will formally replace $\wedge$ symbols with $\circ$ symbols in each $e^i \wedge e^j$ term with $i < j$.  So, for example,
$$\mathsf{s}(-e^{15} - e^{26} + e^{37}) = -e^1 \circ e^5 - e^2 \circ e^6 + e^3 \circ e^7.$$
\indent Finally, we remark that $\text{Sym}^2_0(V^*)$ decomposes into irreducible $\text{SO}(4)$-modules as
\begin{equation}
\textstyle \text{Sym}^2_0(V^*) = \text{Sym}^2_0(\mathsf{A}) \oplus (\mathsf{A} \otimes \mathsf{C})_{1,3} \oplus (\mathsf{A} \otimes \mathsf{C})_{\mathsf{C}} \oplus \text{Sym}^2_0(\mathsf{C}) \oplus \mathbb{R}E_0,
\label{eq:G2-Sym2}
\end{equation}
where
\begin{equation}
E_0 = \text{diag}(4,4,4, -3,-3,-3,-3) \in \text{Sym}^2_0(V^*). \label{eq:E0Def}
\end{equation}

\subsubsection{Decomposition of $3$-Forms on $V^*$}

\indent \indent We now turn to $\Lambda^3(V^*)$.  As noted in \S\ref{ssect:G2Prelims}, $\Lambda^3(V^*)$ splits into $\text{G}_2$-irreducible submodules as
\begin{equation}
\Lambda^3(V^*) \cong \Lambda^3_1 \oplus \Lambda^3_7 \oplus \Lambda^3_{27}. \label{eq:G2-3forms-1}
\end{equation}
The summand $\Lambda^3_1 \cong \mathbb{R}$ is $\text{SO}(4)$-irreducible, but the summands $\Lambda^3_7$ and $\Lambda^3_{27}$ are not. \\
\indent On the other hand, using $V^* \cong \mathsf{A} \oplus \mathsf{C}$, we also have the decomposition:
\begin{align}
\Lambda^3(V^*) \cong \Lambda^3(\mathsf{A}) \oplus (\Lambda^2(\mathsf{A}) \otimes \mathsf{C}) \oplus (\mathsf{A} \otimes \Lambda^2_+(\mathsf{C})) \oplus (\mathsf{A} \otimes \Lambda^2_-(\mathsf{C})) \oplus \Lambda^3(\mathsf{C}). \label{eq:G2-3forms-2}
\end{align}
Three of these summands are $\text{SO}(4)$-irreducible, namely $\Lambda^3(\mathsf{A}) \cong \mathsf{V}_{0,0}$ and $\mathsf{A} \otimes \Lambda^2_-(\mathsf{C}) \cong \mathsf{V}_{2,2}$ and $\Lambda^3(\mathsf{C}) \cong \mathsf{V}_{1,1}$.  Meanwhile, the second and third summands $\Lambda^2(\mathsf{A}) \otimes \mathsf{C}$ and $\mathsf{A} \otimes \Lambda^2_+(\mathsf{C})$ are not.

As in \S\ref{sssect:SO4subG2} above, we will refine both (\ref{eq:G2-3forms-1}) and (\ref{eq:G2-3forms-2}) into $\text{SO}(4)$-irreducible submodules, though only the refinement of (\ref{eq:G2-3forms-1}) will be used in this work.  We begin with (\ref{eq:G2-3forms-1}).

\begin{defn}
	Recall the isomorphism $\mathsf{i} \colon \text{Sym}^2_0(V^*) \to \Lambda^3_{27}$ of (\ref{eq:G2-imap}) and recall the $\text{SO}(4)$-irreducible splitting of $\text{Sym}^2_0(V^*)$ given in (\ref{eq:G2-Sym2}).  We define
	\begin{align*}
	(\Lambda^3_7)_{\mathsf{A}} & := \{\ast(\alpha \wedge \phi_0) \colon \alpha \in \mathsf{A}\} & (\Lambda^3_{27})_{0,0} & := \mathsf{i}(E_0) & (\Lambda^3_{27})_{1,3} & := \mathsf{i}( (\mathsf{A} \otimes \mathsf{C})_{1,3} )  \\
	(\Lambda^3_7)_{\mathsf{C}} & := \{\ast(\alpha \wedge \phi_0) \colon \alpha \in \mathsf{C}\} & (\Lambda^3_{27})_{0,4} & := \mathsf{i}(\text{Sym}^2_0(\mathsf{A})) & (\Lambda^3_{27})_{\mathsf{C}} & := \mathsf{i}( (\mathsf{A} \otimes \mathsf{C})_{\mathsf{C}} ) \\
	& & (\Lambda^3_{27})_{2,2} & := \mathsf{i}(\text{Sym}^2_0(\mathsf{C}))
	\end{align*}
\end{defn}

\begin{lem}\label{lem:G2ThrFormDecomp} 
	The decompositions
	\begin{align*}
	\Lambda^3_7 & = (\Lambda^3_7)_{\mathsf{A}} \oplus (\Lambda^3_7)_{\mathsf{C}} \\
	\Lambda^3_{27} & = (\Lambda^3_{27})_{0,0} \oplus (\Lambda^3_{27})_{0,4} \oplus (\Lambda^3_{27})_{2,2} \oplus (\Lambda^3_{27})_{1,3} \oplus (\Lambda^3_{27})_{\mathsf{C}}
	\end{align*}
	consist of $\text{SO}(4)$-irreducible submodules.
\end{lem}

\begin{defn}
	Recall the isomorphisms $\mathsf{i} \colon \text{Sym}^2_0(V^*) \to \Lambda^3_{27}$ of (\ref{eq:G2-imap}) and $\mathsf{s} \colon (\Lambda^2_7)_{\mathsf{C}} \to (\mathsf{A} \otimes \mathsf{C})_{\mathsf{C}}$ of (\ref{eq:smap}). We define $\dagger \colon (\Lambda^3_{27})_{0,0} \to \mathbb{R}$ to be the unique vector space isomorphism for which
	\begin{equation}
	[\mathsf{i}(E_0)]^\dagger = 4 \sqrt{42} \label{eq:G2-DagIsom}
	\end{equation}
	where $E_0$ is as in (\ref{eq:E0Def}).  The map $\dagger$ is an isometry (due to the choice of $4\sqrt{42}$) with respect to the inner products (\ref{eq:G2-InnerProd}). \\
	\indent We will also need the composition of $\text{SO}(4)$-module isomorphisms
	\begin{align*}
	\mathsf{C}^\sharp & \to (\Lambda^2_7)_{\mathsf{C}} \to (\Lambda^3_{27})_{\mathsf{C}} \\
	X & \mapsto \iota_X\phi_0 \mapsto \textstyle \frac{1}{2\sqrt{3}}\,(\mathsf{i} \circ \mathsf{s})(\iota_X\phi_0).
	\end{align*}
	This map is an isometry due to the factor of $\frac{1}{2\sqrt{3}}$.  We denote the inverse of this isometric isomorphism by
	\begin{equation}
	\ddag \colon (\Lambda^3_{27})_{\mathsf{C}} \to \mathsf{C}^\sharp \label{eq:G2-DDagIsom}
	\end{equation}
\end{defn}

\begin{rmk}
	Extend the isomorphism $L_{\mathsf{C}} \colon \mathsf{A} \to \Lambda^2_+(\mathsf{C})$ to an isomorphism $L_{\mathsf{C}} \colon \mathsf{A} \otimes \mathsf{A} \to \mathsf{A} \otimes \Lambda^2_+(\mathsf{C})$ by the identity on the first $\mathsf{A}$-factor, and split $\mathsf{A} \otimes \mathsf{A} = \mathbb{R} \oplus \text{Sym}^2_0(\mathsf{A}) \oplus \Lambda^2(\mathsf{A})$.  Extend the Hodge star operator $\ast_{\mathsf{A}} \colon \mathsf{A} \to \Lambda^2(\mathsf{A})$ to an isomorphism $\ast_{\mathsf{A}} \colon \mathsf{A} \otimes \mathsf{C} \to \Lambda^2(\mathsf{A}) \otimes \mathsf{C}$ by the identity on the $\mathsf{C}$-factor, and recall the decomposition $\mathsf{A} \otimes \mathsf{C} = (\Lambda^2_7)_{\mathsf{C}} \oplus (\Lambda^2_{14})_{1,3}$. \\
	\indent Defining
	\begin{align*}
	(\Lambda^2(\mathsf{A}) \otimes \mathsf{C})_{\mathsf{C}} & := \ast_{\mathsf{A}}[ (\Lambda^2_7)_{\mathsf{C}} ]  & (\mathsf{A} \otimes \Lambda^2_+(\mathsf{C}))_{0,0} & := L_{\mathsf{C}}( \mathbb{R}) \\
	(\Lambda^2(\mathsf{A}) \otimes \mathsf{C})_{1,3} & := \ast_{\mathsf{A}}[ (\Lambda^2_{14})_{1,3} ] &  (\mathsf{A} \otimes \Lambda^2_+(\mathsf{C}))_{0,4} & := L_{\mathsf{C}}(\text{Sym}^2_0(\mathsf{A}) ) \\ 
	& & (\mathsf{A} \otimes \Lambda^2_+(\mathsf{C}))_{\mathsf{A}} & := L_{\mathsf{C}}( \Lambda^2(\mathsf{A}) ) 
	\end{align*}
	we obtain decompositions
	\begin{align*}
	\Lambda^2(\mathsf{A}) \otimes \mathsf{C} & =  (\Lambda^2(\mathsf{A}) \otimes \mathsf{C})_{\mathsf{C}} \oplus (\Lambda^2(\mathsf{A}) \otimes \mathsf{C})_{1,3} \\
	\mathsf{A} \otimes \Lambda^2_+(\mathsf{C}) & = (\mathsf{A} \otimes \Lambda^2_+(\mathsf{C}))_{0,0}  \oplus (\mathsf{A} \otimes \Lambda^2_+(\mathsf{C}))_{0,4} \oplus (\mathsf{A} \otimes \Lambda^2_+(\mathsf{C}))_{\mathsf{A}}
	\end{align*}
	consisting of $\text{SO}(4)$-irreducible submodules.
\end{rmk}

\begin{rmk}
	The reader can check that some of the above submodules of $\Lambda^3(V^*)$ are, in fact, equal to one another.  Namely, we have the equalities
	\begin{align*}
	\mathsf{A} \otimes \Lambda^2_-(\mathsf{C}) & = (\Lambda^3_{27})_{2,2} &   (\mathsf{A} \otimes \Lambda^2_+(\mathsf{C}))_{0,4} & = (\Lambda^3_{27})_{0,4} \\
	(\Lambda^2(\mathsf{A}) \otimes \mathsf{C})_{1,3} & = (\Lambda^3_{27})_{1,3} & (\mathsf{A} \otimes \Lambda^2_+(\mathsf{C}))_{\mathsf{A}} & = (\Lambda^3_7)_{\mathsf{A}}
	\end{align*}
\end{rmk}


\subsection{The Refined Torsion Forms}\label{ssect:G2RefTors}

\indent \indent Let $(M^7, \varphi)$ be a $7$-manifold equipped with a $\text{G}_2$-structure $\varphi$.  Fix a point $x \in M$, choose an arbitrary associative $3$-plane $\mathsf{A}^\sharp \subset T_xM$, and let $\mathsf{C}^\sharp \subset T_xM$ denote its orthogonal coassociative $4$-plane.  Our purpose in this section is to understand how the torsion of the $\text{G}_2$-structure decomposes with respect to the splitting
$$T_xM = \mathsf{A}^\sharp \oplus \mathsf{C}^\sharp.$$
\indent In \S\ref{sssect:G2reftorsLocFrame}, we use the decompositions of Lemmas \ref{lem:G22FormDecomp} and \ref{lem:G2ThrFormDecomp} to break the torsion forms $\tau_0, \tau_1, \tau_2, \tau_3$ into $\text{SO}(4)$-irreducible pieces called \textit{refined torsion forms}.  Separately, in \S\ref{sssect:G2TorsFuns}, we set up the $\text{G}_2$-coframe bundle $\pi \colon F_{\text{G}_2} \to M$ following \cite{Bryant:2006aa}, repackaging the original $\text{G}_2$ torsion forms $\tau_0, \tau_1, \tau_2, \tau_3$ as a matrix-valued function
$$T = (T_{ij}) \colon F_{\text{G}_2} \to \text{Mat}_{7 \times 7}(\mathbb{R}) \simeq \mathbb{R}^{49}.$$
Finally, in \S\ref{sssect:G2DecompTors}, we express the functions $T_{ij}$ in terms of the (pullbacks of the) refined torsion forms.

\subsubsection{The Refined Torsion Forms in a Local $\text{SO}(4)$-Frame}\label{sssect:G2reftorsLocFrame}

\indent \indent Fix $x \in M$ and split $T_x^*M = \mathsf{A} \oplus \mathsf{C}$ as above.  All of our calculations in this subsection will be done pointwise, and we will suppress reference to $x \in M$.  By Lemmas \ref{lem:G22FormDecomp} and \ref{lem:G2ThrFormDecomp}, the torsion forms $\tau_0, \tau_1, \tau_2, \tau_3$ decompose into $\text{SO}(4)$-irreducible pieces as
\begin{subequations} \label{eq:G2-TorsDecomp-1}
\begin{align}
\tau_0 & = \tau_0 \\
\tau_1 & = (\tau_1)_{\mathsf{A}} + (\tau_1)_{\mathsf{C}} \\
\tau_2 & = (\tau_2)_{\mathsf{A}} + (\tau_2)_{1,3} + (\tau_2)_{2,0} \\
\tau_3 & = (\tau_3)_{0,0} + (\tau_3)_{0,4} + (\tau_3)_{2,2} + (\tau_3)_{1,3} + (\tau_3)_{\mathsf{C}}
\end{align}
\end{subequations}
where here
\begin{align*}
(\tau_1)_{\mathsf{A}} & \in \mathsf{A} & (\tau_2)_{\mathsf{A}} & \in (\Lambda^2_{14})_{\mathsf{A}} & (\tau_3)_{0,0} \in (\Lambda^3_{27})_{0,0} \\
(\tau_1)_{\mathsf{C}} & \in \mathsf{C} & (\tau_2)_{1,3} & \in (\Lambda^2_{14})_{1,3} & (\tau_3)_{0,4} \in (\Lambda^3_{27})_{0,4} \\
 &  & (\tau_2)_{2,0} & \in (\Lambda^2_{14})_{2,0} & (\tau_3)_{2,2} \in (\Lambda^3_{27})_{2,2} \\
 &  & &  & (\tau_3)_{1,3} \in (\Lambda^3_{27})_{1,3} \\
 &  & &  & (\tau_3)_{\mathsf{C}} \in (\Lambda^3_{27})_{\mathsf{C}}
  \end{align*}
We refer to $\tau_0, (\tau_1)_{\mathsf{A}}, (\tau_1)_{\mathsf{C}}, \ldots, (\tau_3)_{\mathsf{C}}$ as the \textit{refined torsion forms} of the $\text{G}_2$-structure at $x$ \textit{relative to the splitting $T_x^*M = \mathsf{A} \oplus \mathsf{C}$}. \\


 
\indent We now move to express the refined torsion forms in terms of a local $\text{SO}(4)$-frame, for which we will need explicit bases of $(\Lambda^2_{14})_{\mathsf{A}}, \ldots, (\Lambda^3_{27})_{\mathsf{C}}$.  To that end, let $\{e_1, \ldots, e_7\}$ be an orthonormal basis for $T_xM$ for which $\mathsf{A}^\sharp = \text{span}(e_1, e_2, e_3)$ and $\mathsf{C}^\sharp = \text{span}(e_4, e_5, e_6, e_7)$.  Let $\{e^1, \ldots, e^7\}$ denote the dual basis for $T_x^*M$. \\

\noindent \textbf{Index Ranges:} We will employ the following index ranges: $1 \leq p,q \leq 3$ and $4 \leq \alpha, \beta \leq 7$ and $1 \leq i,j,k,\ell,m \leq 7$ and $1 \leq \delta \leq 8$ and $1 \leq a \leq 5$. 

\begin{defn}
	Define the $2$-forms
	\begin{align*}
	\Upsilon_1 & = e^{45} + e^{67} & \Omega_1 & = e^{45} - e^{67} & \Delta_1 & = e^{17} + e^{24} & \Delta_5 & = e^{16} + e^{34} \\
	\Upsilon_2 & = e^{46} - e^{57} & \Omega_2 & = e^{46} + e^{57} & \Delta_2 & = e^{16} + e^{25} & \Delta_6 & = -e^{17} + e^{35} \\
	\Upsilon_3 & = -e^{47} - e^{56} & \Omega_3 & = e^{47} - e^{56} & \Delta_3 & = -e^{15} + e^{26} & \Delta_7 & = -e^{14} + e^{36} \\
	& & & & \Delta_4 & = -e^{14} + e^{36} & \Delta_8 & = e^{15} + e^{37}
	\end{align*}
	We also define
	$$\Gamma_p = 2\ast_{\mathsf{A}}\!e^p - \Upsilon_p$$
	(no summation).
\end{defn}

\begin{lem}\label{lem:G22FormBase} 
	 We have that:
	 \begin{enumerate}[label=(\alph*),nosep]
	 	\item $\{\Gamma_1, \Gamma_2, \Gamma_3\}$ is a basis of $(\Lambda^2_{14})_{\mathsf{A}}$.
	 	\item $\{\Delta_1, \ldots, \Delta_8\}$ is a basis of $(\Lambda^2_{14})_{1,3}$.
	 	\item $\{\Omega_1, \Omega_2, \Omega_3\}$ is a basis of $(\Lambda^2_{14})_{2,0}$. 
	 \end{enumerate}
\end{lem}

\begin{defn}
	Define the $3$-forms
	\begin{align*}
	\phi_{\mathsf{A}} & = e^{123} & \lambda_{pq} & = e^p \wedge \Omega_q \\
	\phi_{\mathsf{C}} & = \textstyle \sum e^p \wedge \Upsilon_p & \nu_\alpha & = (\mathsf{i} \circ \mathsf{s})(\iota_{e_\alpha}\phi_0) 
	\end{align*}
	and
	\begin{align*}
	\kappa_1 & = e^1 \wedge \Upsilon_2 - e^2 \wedge \Upsilon_1 & \mu_1 & = e^{237} + e^{314} & \mu_5 & = e^{236} + e^{124} \\
	\kappa_2 & = e^1 \wedge \Upsilon_3 + e^3 \wedge \Upsilon_1 & \mu_2 & = e^{236} + e^{315} & \mu_6 & = -e^{237} + e^{125} \\
	\kappa_3 & = e^2 \wedge \Upsilon_3 + e^3 \wedge \Upsilon_2 & \mu_3 & = -e^{235} + e^{316} & \mu_7 & = -e^{234} + e^{126} \\
	\kappa_4 & = e^1 \wedge \Upsilon_1 - e^2 \wedge \Upsilon_2 & \mu_4 & = -e^{234} + e^{317} & \mu_8 & = e^{235} + e^{127} \\
	\kappa_5 & = e^2 \wedge \Upsilon_2 - e^3 \wedge \Upsilon_3.
	\end{align*}
	Note that $\varphi = \phi_{\mathsf{A}} + \phi_{\mathsf{C}}$.
\end{defn}

\begin{lem}\label{lem:G23FormBase} 
	We have that:
	\begin{enumerate}[label=(\alph*),nosep]
		\item $\{6\phi_{\mathsf{A}} - \phi_{\mathsf{C}}\}$ is a basis of $(\Lambda^3_{27})_{0,0}$
		\item $\{\kappa_1, \kappa_2, \kappa_3, \kappa_4, \kappa_5\}$ is a basis of $(\Lambda^3_{27})_{0,4}$
		\item $\{\lambda_{pq} \colon 1 \leq p,q \leq 3\}$ is a basis of $(\Lambda^3_{27})_{2,2}$
		\item $\{\mu_1, \ldots, \mu_8\}$ is basis of $(\Lambda^3_{27})_{1,3}$
		\item $\{\nu_4, \nu_5, \nu_6, \nu_7\}$ is a basis of $(\Lambda^3_{27})_{\mathsf{C}}$
	\end{enumerate}
\end{lem}

\indent We now express $(\tau_1)_{\mathsf{A}}, (\tau_1)_{\mathsf{C}}$, etc., in terms of the above bases.  That is, we define functions $A_p, B_\alpha$ and $C_p$, $D_\delta, E_p$ and $F, G_a, J_{pq}, L_\delta, M_\alpha$ by:
\begin{subequations} \label{eq:G2-TorsDecomp-2}
\begin{align}
(\tau_1)_{\mathsf{A}} & = 6A_p\, e^p & (\tau_2)_{\mathsf{A}} & = 12C_p\,\Gamma_p & (\tau_3)_{0,0} & = 12F\,(6\phi_{\mathsf{A}} - \phi_{\mathsf{C}}) \\
(\tau_1)_{\mathsf{C}} & = 6B_\alpha\, e^\alpha & (\tau_2)_{1,3} & = 12D_\delta\,\Delta_\delta & (\tau_3)_{0,4} & = 6G_a\, \kappa_a \\
& & (\tau_2)_{2,0} & = 12E_p\,\Omega_p & (\tau_3)_{2,2} & = 12J_{pq}\, \lambda_{pq} \\
& & & & (\tau_3)_{1,3} & = 12L_\delta\,\mu_\delta \\
& & & & (\tau_3)_{\mathsf{C}} & = 6M_\alpha\,\nu_\alpha
\end{align}
\end{subequations}
The various factors of $6$ and $12$ are included simply for the sake of clearing future denominators.

Note that the bases of Lemmas \ref{lem:G22FormBase} and \ref{lem:G23FormBase} are orthogonal but \textit{not} orthonormal with respect to the inner product (\ref{eq:G2-InnerProd}) on $\Lambda^k(V^*)$.  Indeed, we have:
\begin{align*}
\Vert \Gamma_p \Vert & = \sqrt{6} & \Vert \Omega_p \Vert & = \sqrt{2} & \Vert \mu_\delta \Vert & = \sqrt{2} & \Vert \kappa_a \Vert & = 2 \\
\Vert \Delta_\delta \Vert & = \sqrt{2} & \Vert 6\phi_{\mathsf{A}} - \phi_{\mathsf{C}} \Vert & = \sqrt{42} & \Vert \nu_\alpha \Vert & = 2\sqrt{3} &  \Vert \lambda_{pq} \Vert & = \sqrt{2}
\end{align*}
Thus, in terms of the isometric isomorphisms (\ref{eq:G2-Sharp}), (\ref{eq:G2-NatIsom}), (\ref{eq:G2-DagIsom}), (\ref{eq:G2-DDagIsom}) of \S\ref{ssect:SO4reps}, we have:
\begin{subequations} \label{eq:G2-NormalizedTors}
\begin{align}
[(\tau_1)_{\mathsf{A}}]^\sharp & = 6A_p e_p & [(\tau_2)_{\mathsf{A}}]^\natural & = 12\sqrt{6}\,C_pe_p & [(\tau_3)_{0,0}]^\dagger & = 12\sqrt{42}\,F \\
[(\tau_1)_{\mathsf{C}}]^\sharp & = 6B_\alpha e_\alpha & & & [(\tau_3)_{\mathsf{C}}]^\ddag & = 12\sqrt{3}\,M_\alpha e_\alpha
\end{align}
\end{subequations}
We will need these for our calculations in \S\ref{ssect:MCAssoc} and \S\ref{ssect:MCCoassoc}.

\subsubsection{The Torsion Functions $T_{ij}$}\label{sssect:G2TorsFuns}

\indent \indent Let $(M^7, \varphi)$ be a $7$-manifold with a $\text{G}_2$-structure $\varphi$, and let $g_\varphi$ denote the underlying Riemannian metric.  Let $F_{\text{SO}(7)} \to M$ denote the oriented orthonormal coframe bundle of $g_\varphi$, and let $\omega = (\omega^1, \ldots, \omega^7) \in \Omega^1(F_{\text{SO}(7)}; \mathbb{R}^7)$ denote the tautological $1$-form.  By the Fundamental Lemma of Riemannian Geometry, there exists a unique $1$-form $\psi \in \Omega^1(F_{\text{SO}(7)}; \mathfrak{so}(7))$, the Levi-Civita connection form of $g_\varphi$, satisfying the First Structure Equation
\begin{equation*}
d\omega = -\psi \wedge \omega.
\end{equation*}
\indent Let $\pi \colon F_{\text{G}_2} \to M$ denote the $\text{G}_2$-coframe bundle of $M$.  Restricted to $F_{\text{G}_2} \subset F_{\text{SO}(7)}$, the Levi-Civita $1$-form $\psi$ is no longer a connection $1$-form in general.  Indeed, according to the splitting $\mathfrak{so}(7) = \mathfrak{g}_2 \oplus \mathbb{R}^7$, we have the decomposition
$$\psi = \theta + 2\gamma,$$
where $\theta = (\theta_{ij}) \in \Omega^1(F_{\text{G}_2}; \mathfrak{g}_2)$ is a connection $1$-form (the so-called \textit{natural connection} of the $\text{G}_2$-structure $\varphi$) and $\gamma \in \Omega^1(F_{\text{G}_2}; \mathbb{R}^7)$ is a $\pi$-semi-basic $1$-form.  Here, we are viewing $\mathbb{R}^7 = \{(\epsilon_{ijk}v_k) \in \mathfrak{so}(7) \colon  (v_1, \ldots, v_7) \in \mathbb{R}^7\}$, so that $\gamma$ takes the form
$$\gamma  = \left[ \begin{array}{c c c | c c c c}
0 & \gamma_3 & -\gamma_2 & \gamma_5 & -\gamma_4 & \gamma_7 & -\gamma_6 \\
-\gamma_3 & 0 & \gamma_1 & \gamma_6 & -\gamma_7 & -\gamma_4 & \gamma_5 \\
\gamma_2 & -\gamma_1 & 0 & -\gamma_7 & -\gamma_6 & \gamma_5 & \gamma_4 \\ \hline
-\gamma_5 & -\gamma_6 & \gamma_7 & 0 & \gamma_1 & \gamma_2 & -\gamma_3 \\
\gamma_4 & \gamma_7 & \gamma_6 & -\gamma_1 & 0 & -\gamma_3 & -\gamma_2 \\
-\gamma_7 & \gamma_4 & -\gamma_5 & -\gamma_2 & \gamma_3 & 0 & \gamma_1 \\
\gamma_6 & -\gamma_5 & -\gamma_4 & \gamma_3 & \gamma_2 & -\gamma_1 & 0
\end{array}\right]\!.$$
Since $\gamma$ is $\pi$-semibasic, we may write
$$\gamma_i = T_{ij}\omega^j$$
for some matrix-valued function $T = (T_{ij}) \colon F_{\text{G}_2} \to \text{Mat}_{7\times 7}(\mathbb{R})$.  The $1$-form $\gamma$, and hence the functions $T_{ij}$, encodes the torsion of the $\text{G}_2$-structure.  In this notation, the first structure equation reads
\begin{equation} \label{eq:G2-FirstStrEqn}
d\omega_i = -(\theta_{ij} + 2\epsilon_{ijk}\gamma_k) \wedge \omega_j 
\end{equation}
\begin{rmk}
	The reader may wonder how the functions $T_{ij}$ are related to the forms $\tau_0, \tau_1, \tau_2, \tau_3$.  In \cite{Bryant:2006aa}, Bryant expresses the torsion forms $\tau_0, \tau_1, \tau_2, \tau_3$ in terms of $T_{ij}$ as:
	\begin{align*}
	\pi^*(\tau_0) & = \textstyle \frac{24}{7}T_{ii} \\
	\pi^*(\tau_1) & = \textstyle \epsilon_{ijk}T_{ij}\,\omega_k \\
	\pi^*(\tau_2) & = \textstyle 4T_{ij}\,\omega_i \wedge \omega_j - \epsilon_{ijk\ell}T_{ij}\,\omega_k \wedge \omega_\ell \\
	\pi^*(\tau_3) & = \textstyle -\frac{3}{2}\epsilon_{ik\ell}(T_{ij} + T_{ji})\,\omega_{jk\ell} + \frac{18}{7}T_{ii}\sigma.
	\end{align*}
	In the next section, we will exhibit a sort of inverse to this, expressing the $T_{ij}$ in terms of the refined torsion forms $\pi^*(\tau_0)$, $\pi^*((\tau_1)_{\mathsf{A}}), \ldots, \pi^*((\tau_3)_{\mathsf{C}})$.
\end{rmk}

\subsubsection{Decomposition of the Torsion Functions}\label{sssect:G2DecompTors}

\indent \indent For our computations in \S\ref{ssect:MCAssoc} and \S\ref{ssect:MCCoassoc}, we will need to express the torsion functions $T_{ij}$ in terms of the functions $A_p, B_\alpha, \ldots, L_\delta, M_\alpha$.  To this end, we will continue to work on the total space of the $\text{G}_2$-coframe bundle $\pi \colon F_{\text{G}_2} \to M$, pulling back all of the quantities defined on $M$ to $F_{\text{G}_2}$.  Following common convention, we systematically omit $\pi^*$ from the notation, so that (for example) $\pi^*(\tau_0)$ will simply be denoted $\tau_0$, etc.  Note, however, that $\pi^*(e^j) = \omega_j$. \\

\indent To begin, recall that the torsion forms $\tau_0, \tau_1, \tau_2, \tau_3$ satisfy
\begin{align*}
d\varphi & = \tau_0 \ast\!\varphi + 3\tau_1 \wedge \varphi + \ast\tau_3 \\
d\ast\! \varphi & = \ \ \ \ \ \ \ \ \ \ 4\tau_1 \wedge \ast \varphi + \tau_2 \wedge \varphi.
\end{align*}
Into the left-hand sides, we substitute $\varphi = \frac{1}{6}\epsilon_{ijk}\,\omega^{ijk}$ and $\ast\varphi = \frac{1}{24} \epsilon_{ijk\ell}\, \omega^{ijk\ell}$ and use the first structure equation (\ref{eq:G2-FirstStrEqn}) to obtain
\begin{align*}
\epsilon_{ijk\ell}\,T_{im}\,\omega^{mjk\ell} & = \tau_0 \ast\!\varphi + 3\tau_1 \wedge \varphi + \ast\tau_3 \\
-\epsilon_{ijk}\, T_{\ell m}\,\omega^{m\ell ijk} & = \ \ \ \ \ \ \ \ \ \ 4\tau_1 \wedge \ast \varphi + \tau_2 \wedge \varphi
\end{align*}
Into the right-hand sides, we again substitute $\varphi = \frac{1}{6}\epsilon_{ijk}\,\omega^{ijk}$ and $\ast\varphi = \frac{1}{24} \epsilon_{ijk\ell}\, \omega^{ijk\ell}$, as well as the expansions (\ref{eq:G2-TorsDecomp-1}) and (\ref{eq:G2-TorsDecomp-2}). \\
\indent Upon equating coefficients, we obtain a system of $56 = \binom{7}{4} + \binom{7}{5}$ linear equations relating the $49 = 7^2$ functions $T_{ij}$ on the left side to the $49 = \dim(H^{0,2}(\mathfrak{g}_2))$ functions $\tau_0, A_p, B_\alpha, \ldots, L_\delta, M_\alpha$ on the right side.  One can then use a computer algebra system (we have used M\textsc{aple}) to solve this linear system for the $T_{ij}$. \\
\indent We now exhibit the result, taking advantage of the $\text{SO}(4)$-irreducible splitting
\begin{align*}
\text{Mat}_{7\times 7}(\mathbb{R}) \cong V^* \otimes V^* & \cong (\mathsf{A} \otimes \mathsf{A}) \oplus 2(\mathsf{A} \otimes \mathsf{C}) \oplus (\mathsf{C} \otimes \mathsf{C}) \\
& \cong \left( \Lambda^2(\mathsf{A})  \oplus \text{Sym}^2_0(\mathsf{A}) \oplus \mathbb{R} \right)  \oplus 2\!\left((\mathsf{A} \otimes \mathsf{C})_{1,3} \oplus  (\mathsf{A} \otimes \mathsf{C})_{\mathsf{C}} \right) \\
& \ \ \ \ \ \oplus \left(\Lambda^2_+(\mathsf{C}) \oplus \Lambda^2_-(\mathsf{C}) \oplus \mathbb{R} \oplus \text{Sym}^2_0(\mathsf{C}) \right)
\end{align*}
to highlight the structure of the solution.

We have
\begin{align*}
\frac{1}{2} \begin{bmatrix}
0 & T_{12} - T_{21} & T_{13} - T_{31} \\
T_{21} - T_{12} & 0 & T_{23} - T_{32}  \\
T_{31} - T_{13} & T_{32} - T_{23} & 0
\end{bmatrix} & =  \begin{bmatrix}
0 & A_2 + 2C_3 & -(A_2 + 2C_2) \\
-(A_3 + 2C_3) & 0 & A_1 + 2C_1 \\
A_2 + 2C_2 & -(A_1 + 2C_1) & 0
\end{bmatrix}, \\
\frac{1}{2}\begin{bmatrix}
2T_{11} & T_{12} + T_{21} & T_{13} + T_{31}  \\
T_{21} + T_{12} & 2T_{22} & T_{23} + T_{32}  \\
T_{31} + T_{13} & T_{32} + T_{23} & 2T_{33}
\end{bmatrix} & = - \begin{bmatrix}
G_4 & G_1 & G_2 \\
G_1 & G_5 - G_4 & G_3 \\
G_2 & G_3 & -G_5
\end{bmatrix} + \left(-4F + \frac{1}{24}\tau_0\right)\text{Id}_3,
\end{align*}
corresponding to $\mathsf{A} \otimes \mathsf{A} \cong \Lambda^2(\mathsf{A})  \oplus \text{Sym}^2_0(\mathsf{A}) \oplus \mathbb{R}$ and
\begin{align*}
\frac{1}{2}\begin{bmatrix}
T_{41} + T_{14} & T_{42} + T_{24} & T_{43} + T_{34} \\
T_{51} + T_{15} & T_{52} + T_{25} & T_{53} + T_{35} \\
T_{61} + T_{16} & T_{62} + T_{26} & T_{63} + T_{36} \\
T_{71} + T_{17} & T_{72} + T_{27} & T_{73} + T_{37}
\end{bmatrix} & =
\begin{bmatrix}
L_4 + L_7 & -L_1 & -L_5 \\
L_3 - L_8 & -L_2 & -L_6 \\
-L_2 - L_5 & -L_3 & -L_7 \\
-L_1 + L_6 & -L_4 & -L_8
\end{bmatrix} +
\begin{bmatrix}
-M_5 & -M_6 & M_7 \\
M_4 & M_7 & M_6 \\
-M_7 & M_4 & -M_5 \\
M_6 & -M_5 & -M_4
\end{bmatrix}
\end{align*}
and
\begin{align*}
\frac{1}{2}\begin{bmatrix}
T_{41} - T_{14} & T_{42} - T_{24} & T_{43} - T_{34} \\
T_{51} - T_{15} & T_{52} - T_{25} & T_{53} - T_{35} \\
T_{61} - T_{16} & T_{62} - T_{26} & T_{63} - T_{36} \\
T_{71} - T_{17} & T_{72} - T_{27} & T_{73} - T_{37}
\end{bmatrix} & =
\begin{bmatrix}
D_4 + D_7 & -D_1 & -D_5 \\
D_3 - D_8 & -D_2 & -D_6 \\
-D_2 - D_5 & -D_3 & -D_7 \\
-D_1 + D_6 & -D_4 & -D_8
\end{bmatrix} +
\begin{bmatrix}
-B_5 & -B_6 & B_7 \\
B_4 & B_7 & B_6 \\
-B_7 & B_4 & -B_5 \\
B_6 & -B_5 & -B_4
\end{bmatrix}
\end{align*}
corresponding to $\mathsf{A} \otimes \mathsf{C} \cong (\mathsf{A} \otimes \mathsf{C})_{1,3} \oplus  (\mathsf{A} \otimes \mathsf{C})_{\mathsf{C}}$, and
\begin{align*}
\frac{1}{2} \begin{bmatrix}
0 & T_{45} - T_{54} & T_{46} - T_{64} & T_{47} - T_{74} \\
T_{54} - T_{45} & 0 & T_{56} - T_{65} & T_{57} - T_{57} \\
T_{64} - T_{46} & T_{65} - T_{56} & 0 & T_{67} - T_{76} \\
T_{74} - T_{47} & T_{75} - T_{57} & T_{76} - T_{67} & 0
\end{bmatrix} & = \begin{bmatrix}
0 & A_1 - C_1 & A_2 - C_2 & -A_3 + C_3 \\
-A_1 + C_1 & 0 & -A_3 + C_3 & -A_2 + C_2 \\
-A_2 - C_2 & A_3 - C_3 & 0 & A_1 - C_1 \\
A_3 - C_3 & A_2 - C_2 & -A_1 + C_1 & 0
\end{bmatrix} \\
& \ \ +
\begin{bmatrix}
0 & E_1 & E_2 & E_3 \\
-E_1 & 0 & -E_3 & E_2 \\
-E_2 & E_3 & 0 & -E_1 \\
-E_3 & -E_2 & E_1 & 0
\end{bmatrix}
\end{align*}
and
\begin{align*}
& \frac{1}{2} \begin{bmatrix}
2T_{44} & T_{45} + T_{54} & T_{46} + T_{64} & T_{47} + T_{74} \\
T_{54} + T_{45} & 2T_{55} & T_{56} + T_{65} & T_{57} + T_{75} \\
T_{64} + T_{46} & T_{65} & 2T_{66} & T_{67} + T_{76} \\
T_{74} + T_{47} & T_{75} + T_{57} & T_{76} + T_{67} & 2T_{77}
\end{bmatrix} = \\
& \ \ \ \ \begin{bmatrix}
- J_{11} -J_{22}  + J_{33} & J_{23} + J_{32} & -J_{13} - J_{31} & J_{12} - J_{21}  \\
J_{23} + J_{32} & - J_{11} + J_{22} - J_{33} & -J_{21} - J_{12} & -J_{13} + J_{31} \\
-J_{13} - J_{31} & -J_{12} - J_{21} & J_{11} - J_{22} - J_{33} & - J_{23} + J_{32}  \\
J_{12} - J_{21} & -J_{13} + J_{31} & -J_{23} + J_{32} &  J_{11} + J_{22} + J_{33}
\end{bmatrix} + \left(3F + \frac{1}{24}\tau_0\right) \text{Id}_4
\end{align*}
corresponding to $\mathsf{C} \otimes \mathsf{C} \cong \Lambda^2_+(\mathsf{C}) \oplus \Lambda^2_-(\mathsf{C}) \oplus \mathbb{R} \oplus \text{Sym}^2_0(\mathsf{C})$. \\

\indent The above relations are more than we need for this work.  In fact, we will only make use of the following relations, which can be read off from the above:
\begin{align} \label{eq:G2-TorSol-1}
\epsilon_{\alpha \beta p} T_{\beta p}  = -3(B_\alpha + M_\alpha)
\end{align}
and
\begin{align} \label{eq:G2-TorSol-2}
T_{44} + T_{55} + T_{66} + T_{77} & = 3F + \frac{1}{24}\tau_0 
\end{align}
and
\begin{subequations} \label{eq:G2-TorSol-3}
\begin{align}
-(T_{45} - T_{54}) - (T_{67} - T_{76}) & = -4(A_1 - C_1) \\
(T_{57} - T_{75}) - (T_{46} - T_{64}) & = -4(A_2 - C_2) \\
(T_{47} - T_{74}) + (T_{56} - T_{65}) & = -4(A_3 - C_3).
\end{align}
\end{subequations}


\subsection{Mean Curvature of Associative $3$-Folds}\label{ssect:MCAssoc}

\indent \indent In this subsection, we derive a formula (Theorem \ref{thm:MCAssoc}) for the mean curvature of an associative $3$-fold in an arbitrary $7$-manifold $(M, \varphi)$ with $\text{G}_2$-structure $\varphi$. \\

\indent We continue with the notation of \S\ref{ssect:G2RefTors}, letting $\pi \colon F_{\text{G}_2} \to M$ denote the $\text{G}_2$-coframe bundle of $M$, and $\omega = (\omega_{\mathsf{A}}, \omega_{\mathsf{C}}) \in \Omega^1(F_{\text{G}_2}; \mathsf{A}^\sharp \oplus \mathsf{C}^\sharp)$ denote the tautological $1$-form.  We remind the reader that $\theta = (\theta_{ij}) \in \Omega^1(F_{\text{G}_2}; \mathfrak{g}_2)$ is the natural connection $1$-form, and that $\gamma = (\gamma_{ij}) \in \Omega^1(F_{\text{G}_2}; \mathbb{R}^7)$ is a $\pi$-semibasic $1$-form encoding the torsion of $\varphi$.  We will continue to write $\gamma_{ij} = \epsilon_{ijk}\gamma_k$ and $\gamma_i = T_{ij}\omega^j$ for $T = (T_{ij}) \colon F_{\text{G}_2} \to \text{Mat}_{7 \times 7}(\mathbb{R})$. \\
\indent Let $f \colon \Sigma^3 \to M^7$ denote an immersion of an associative $3$-fold into $M$, and let $f^*(F_{\text{G}_2}) \to \Sigma$ denote the pullback bundle.  Let $B \subset f^*(F_{\text{G}_2})$ denote the subbundle of coframes adapted to $\Sigma$, i.e., the subbundle whose fiber over $x \in \Sigma$ is
$$B|_x = \{u \in f^*(F_{\text{G}_2})|_x \colon u(T_x\Sigma) = \mathsf{A}^\sharp \oplus 0 \}$$
We recall (Proposition \ref{prop:G2ActStab}) that $\text{G}_2$ acts transitively on the set of associative $3$-planes with stabilizer $\text{SO}(4)$, so $B \to \Sigma$ is a well-defined $\text{SO}(4)$-bundle.  Note that on $B$, we have
$$\omega_{\mathsf{C}} = 0.$$
For the rest of \S\ref{ssect:MCAssoc}, all of our calculations will be done on the subbundle $B \subset F_{\text{G}_2}$. \\

\indent We now exploit splitting $T_xM = T_x\Sigma \oplus (T_x\Sigma)^\perp \simeq \mathsf{A}^\sharp \oplus \mathsf{C}^\sharp$ to decompose $\theta$ and $\gamma$ into $\text{SO}(4)$-irreducible pieces.  To decompose the connection $1$-form $\theta \in \Omega^1(B; \mathfrak{g}_2)$, we split
$$\mathfrak{g}_2 \cong [\mathfrak{g}_2 \cap (\Lambda^2(\mathsf{A}) \oplus \Lambda^2_+(\mathsf{C}))] \oplus [\mathfrak{g}_2 \cap (\mathsf{A} \otimes \mathsf{C})] \oplus [\mathfrak{g}_2 \cap \Lambda^2_-(\mathsf{C})],$$
so that $\theta$ takes the block form
$$\theta = \begin{bmatrix} \rho(\zeta) & -\sigma^T \\
\sigma & \zeta + \xi \end{bmatrix} = \left[ \begin{array}{c c c | c c c c}
0 & 2\zeta_3 & -2\zeta_2 & -\sigma_4 - \sigma_7 & -\sigma_3 + \sigma_8 & \sigma_2 + \sigma_5 & \sigma_1 - \sigma_6 \\
-2\zeta_3 & 0 & 2\zeta_1 & \sigma_1 & \sigma_2 & \sigma_3 & \sigma_4 \\
2\zeta_2 & -2\zeta_1 & 0  & \sigma_5 & \sigma_6 & \sigma_7 & \sigma_8 \\ \hline
\sigma_4 + \sigma_7 & -\sigma_1 & -\sigma_5 & 0 & -\zeta_1-\xi_1 & -\zeta_2 + \xi_2 & \zeta_3 + \xi_3 \\
\sigma_3 - \sigma_8 & -\sigma_2 & -\sigma_6 & \zeta_1 + \xi_1 & 0 & \zeta_3 - \xi_3 & \zeta_2 + \xi_2 \\
-\sigma_2 - \sigma_5 & -\sigma_3 & -\sigma_7 & \zeta_2 - \xi_2 & -\zeta_3 + \xi_3 & 0 & -\zeta_1 + \xi_1 \\
-\sigma_1 + \sigma_6 & -\sigma_4 & -\sigma_8 & -\zeta_3 - \xi_3 & -\zeta_2 - \xi_2 & \zeta_1 - \xi_1 & 0
\end{array}\right]\!.$$
Similarly, the $1$-form $\gamma \in \Omega^1(B; \mathbb{R}^7)$ breaks into block form as:
$$\gamma = \begin{bmatrix}
\gamma_{\mathsf{A}} & -(\gamma_{\mathsf{C}})^T \\
\gamma_{\mathsf{C}} & (\gamma_{\mathsf{A}})_+ \\
\end{bmatrix} = \left[ \begin{array}{c c c | c c c c}
0 & \gamma_3 & -\gamma_2 & \gamma_5 & -\gamma_4 & \gamma_7 & -\gamma_6 \\
-\gamma_3 & 0 & \gamma_1 & \gamma_6 & -\gamma_7 & -\gamma_4 & \gamma_5 \\
\gamma_2 & -\gamma_1 & 0 & -\gamma_7 & -\gamma_6 & \gamma_5 & \gamma_4 \\ \hline
-\gamma_5 & -\gamma_6 & \gamma_7 & 0 & \gamma_1 & \gamma_2 & -\gamma_3 \\
\gamma_4 & \gamma_7 & \gamma_6 & -\gamma_1 & 0 & -\gamma_3 & -\gamma_2 \\
-\gamma_7 & \gamma_4 & -\gamma_5 & -\gamma_2 & \gamma_3 & 0 & \gamma_1 \\
\gamma_6 & -\gamma_5 & -\gamma_4 & \gamma_3 & \gamma_2 & -\gamma_1 & 0
\end{array}\right]\!.$$ \\
\indent In this notation, the first structure equation (\ref{eq:G2-FirstStrEqn}) on $B$ reads:
\begin{align*}
d\begin{pmatrix} \omega_{\mathsf{A}} \\ 0 \end{pmatrix} = -\left(\begin{bmatrix} \rho(\zeta) & -\sigma^T \\
\sigma & \zeta + \xi \end{bmatrix} + 2\begin{bmatrix}
\gamma_{\mathsf{A}} & -(\gamma_{\mathsf{C}})^T \\
\gamma_{\mathsf{C}} & (\gamma_{\mathsf{A}})_+ \\
\end{bmatrix} \right) \wedge \begin{pmatrix} \omega_{\mathsf{A}} \\ 0 \end{pmatrix}\!.
\end{align*}
In particular, the second line gives
$$0 = -(\sigma + 2\gamma^{\mathsf{C}}) \wedge \omega_{\mathsf{A}}$$
or in detail,
\begin{equation} \label{eq:Assoc-CondensedStrEqn}
\begin{bmatrix}
\sigma_4 + \sigma_7 & -\sigma_1 & -\sigma_5 \\
\sigma_3 - \sigma_8 & -\sigma_2 & -\sigma_6 \\
-\sigma_2 - \sigma_5 & -\sigma_3 & -\sigma_7 \\
-\sigma_1 + \sigma_6 & -\sigma_4 & -\sigma_8 
\end{bmatrix} \wedge \begin{bmatrix} \omega^1 \\ \omega^2 \\ \omega^3 \end{bmatrix} = -2 \begin{bmatrix}
-\gamma_5 & -\gamma_6 & \gamma_7  \\
\gamma_4 & \gamma_7 & \gamma_6  \\
-\gamma_7 & \gamma_4 & -\gamma_5 \\
\gamma_6 & -\gamma_5 & -\gamma_4 
\end{bmatrix} \wedge \begin{bmatrix} \omega^1 \\ \omega^2 \\ \omega^3 \end{bmatrix}
\end{equation}
Note that on $B$, the $1$-forms $\sigma_\delta$ and $\gamma_\alpha$ are semibasic, and we write
\begin{align*}
\sigma_\delta & = S_{\delta p}\,\omega^p & \gamma_\alpha & = T_{\alpha p}\, \omega^p
\end{align*}
for some function $S = (S_{\delta p}) \colon B \to \mathsf{V}_{1,3} \otimes \mathsf{A}$, recalling our index ranges $1 \leq p \leq 3$ and $4 \leq \alpha \leq 7$ and $1 \leq \delta \leq 8$. \\
\indent Now, the $24$ functions $S_{\delta p}$ and the $12$ functions $T_{\alpha p}$ are not independent: the equation (\ref{eq:Assoc-CondensedStrEqn}) amounts to $12 = 4 \binom{3}{2}$ linear relations among them.  Explicitly:
$$\begin{bmatrix}
S_{13} - S_{52}  & S_{43} + S_{73} + S_{51} & -S_{42} - S_{72} - S_{11}  \\
S_{23} - S_{62}  & S_{33} - S_{83} + S_{61} & -S_{32} + S_{82} - S_{21}  \\
S_{33} - S_{72}  & -S_{23} - S_{53} + S_{71} & S_{22} + S_{52} - S_{31}  \\
S_{43} - S_{82}  & -S_{13} + S_{63} + S_{81} & S_{12} - S_{62} - S_{41}
  \end{bmatrix} 
= -2\begin{bmatrix}
T_{63} + T_{72} & -T_{53} - T_{71} & T_{52} - T_{61} \\
T_{62} - T_{73} & T_{43} - T_{61} & -T_{42} + T_{71} \\
-T_{43} - T_{52} & T_{51} - T_{73} & T_{41} + T_{72} \\
-T_{42} + T_{53} & T_{41} + T_{63} & -T_{51} - T_{62}
\end{bmatrix}
$$
In particular, these relations imply:
\begin{subequations} \label{eq:G2-TorsRel}
\begin{align}
S_{41} + S_{71} - S_{12} - S_{53} & = -4 \epsilon_{4\alpha p} T_{\alpha p} \\
S_{31} - S_{81} - S_{22} - S_{63} & = -4 \epsilon_{5\alpha p} T_{\alpha p} \\
-S_{21} - S_{51} - S_{32} - S_{73} & = -4 \epsilon_{6\alpha p} T_{\alpha p} \\
-S_{11} + S_{61} - S_{42} - S_{83} & = -4 \epsilon_{7\alpha p} T_{\alpha p}
\end{align}
\end{subequations}
With these calculations in place, we may finally compute the mean curvature of an associative $3$-fold:
\begin{thm}\label{thm:MCAssoc} 
	Let $\Sigma \subset M$ be an associative $3$-fold immersed in a $7$-manifold $M$ equipped with a $\text{G}_2$-structure.  Then the mean curvature vector $H$ of $\Sigma$ is given by
	$$H = -3 [(\tau_1)_{\mathsf{C}}]^\sharp - \frac{\sqrt{3}}{2} \left[(\tau_3)_{\mathsf{C}}\right]^\ddag.$$
	In particular, the largest torsion class of $\text{G}_2$-structures $\varphi$ for which every associative $3$-fold is minimal is $W_1 \oplus W_{14} = W_1 \cup W_{14}$, i.e., the class for which $d\varphi = \lambda \ast\! \varphi$ for some $\lambda \in \mathbb{R}$.
\end{thm}
\begin{proof}
	The mean curvature vector may be computed as follows:
	\begin{align}
	\begin{bmatrix} H_4 \\ H_5 \\ H_6 \\ H_7 \end{bmatrix} \omega^{123} & = \begin{bmatrix}
	\psi_{41} & \psi_{42} & \psi_{43} \\
	\psi_{51} & \psi_{52} & \psi_{53} \\
	\psi_{61} & \psi_{62} & \psi_{63} \\
	\psi_{71} & \psi_{72} & \psi_{73}
	\end{bmatrix} \wedge \begin{bmatrix} \omega^{23} \\ \omega^{31} \\ \omega^{12} 
	\end{bmatrix}  \notag \\
	& = \begin{bmatrix}
	\sigma_4 + \sigma_7 & -\sigma_1 & -\sigma_5 \\
	\sigma_3 - \sigma_8 & -\sigma_2 & -\sigma_6 \\
	-\sigma_2 - \sigma_5 & -\sigma_3 & -\sigma_7 \\
	-\sigma_1 + \sigma_6 & -\sigma_4 & -\sigma_8 
	\end{bmatrix} \wedge \begin{bmatrix} \omega^{23} \\ \omega^{31} \\ \omega^{12} 
	\end{bmatrix} +\,
	2\!\begin{bmatrix}
	-\gamma_5 & -\gamma_6 & \gamma_7  \\
	\gamma_4 & \gamma_7 & \gamma_6  \\
	-\gamma_7 & \gamma_4 & -\gamma_5 \\
	\gamma_6 & -\gamma_5 & -\gamma_4 
	\end{bmatrix}  \wedge \begin{bmatrix} \omega^{23} \\ \omega^{31} \\ \omega^{12} 
	\end{bmatrix} \label{eq:Assoc-MeanCurv}
	\end{align}
	To evaluate the first term in (\ref{eq:Assoc-MeanCurv}), we substitute $\sigma_\delta = S_{\delta p}\omega^p$, followed by (\ref{eq:G2-TorsRel}), and finally (\ref{eq:G2-TorSol-1}), to obtain:
	\begin{align*}
	\begin{bmatrix}
	\sigma_4 + \sigma_7 & -\sigma_1 & -\sigma_5 \\
	\sigma_3 - \sigma_8 & -\sigma_2 & -\sigma_6 \\
	-\sigma_2 - \sigma_5 & -\sigma_3 & -\sigma_7 \\
	-\sigma_1 + \sigma_6 & -\sigma_4 & -\sigma_8 
	\end{bmatrix} \wedge \begin{bmatrix} \omega^{23} \\ \omega^{31} \\ \omega^{12} 
	\end{bmatrix} & = \begin{bmatrix}
	S_{41} + S_{71} - S_{12} - S_{53} \\
	S_{31} - S_{81} - S_{22} - S_{63} \\
	-S_{21} - S_{51} - S_{32} - S_{73} \\
	-S_{11} + S_{61} - S_{42} - S_{83}
	\end{bmatrix}
	\omega^{123} \\
	& = -4 \begin{bmatrix} \epsilon_{4\alpha p}T_{\alpha p} \\ \epsilon_{5\alpha p}T_{\alpha p} \\ \epsilon_{6\alpha p}T_{\alpha p} \\ \epsilon_{7\alpha p}T_{\alpha p} \end{bmatrix} \omega^{123} =
	-12\begin{bmatrix}
	B_4 + M_4 \\
	B_5 + M_5 \\
	B_6 + M_6 \\
	B_7 + M_7
	\end{bmatrix}\omega^{123}
	\end{align*}
	Similarly, to evaluate the second term in (\ref{eq:Assoc-MeanCurv}), we substitute $\gamma_\alpha = T_{\alpha p}\omega^p$ followed by (\ref{eq:G2-TorSol-1}) to obtain:
	\begin{align*}
	2\begin{bmatrix}
	-\gamma_5 & -\gamma_6 & \gamma_7  \\
	\gamma_4 & \gamma_7 & \gamma_6  \\
	-\gamma_7 & \gamma_4 & -\gamma_5 \\
	\gamma_6 & -\gamma_5 & -\gamma_4 
	\end{bmatrix}  \wedge \begin{bmatrix} \omega^{23} \\ \omega^{31} \\ \omega^{12} 
	\end{bmatrix}
	& = -2 \begin{bmatrix} \epsilon_{4\alpha p}T_{\alpha p} \\ \epsilon_{5\alpha p}T_{\alpha p} \\ \epsilon_{6\alpha p}T_{\alpha p} \\ \epsilon_{7\alpha p}T_{\alpha p} \end{bmatrix} \omega^{123} =
	-6\begin{bmatrix}
	B_4 + M_4 \\
	B_5 + M_5 \\
	B_6 + M_6 \\
	B_7 + M_7
	\end{bmatrix} \omega^{123}
	\end{align*}
	We conclude that
	$$H_\alpha = -18\,B_\alpha - 18\,M_\alpha,$$
	and so (\ref{eq:G2-NormalizedTors}) yields
	$$H = -3 [(\tau_1)_{\mathsf{C}}]^\sharp - \frac{\sqrt{3}}{2} \left[(\tau_3)_{\mathsf{C}}\right]^\ddag.$$
	\noindent In particular, the largest torsion class for which $H = 0$ for all associatives is the one for which $\tau_1 = \tau_3 = 0$, which is $W_1 \oplus W_{14} = W_1 \cup W_{14}$.
\end{proof}


\subsection{Mean Curvature of Coassociative $4$-Folds}\label{ssect:MCCoassoc}

\indent \indent In this subsection, we derive a formula (Theorem \ref{thm:MCCoass}) for the mean curvature of a coassociative $4$-fold in an arbitrary $7$-manifold $(M, \varphi)$ with $\text{G}_2$-structure $\varphi$.  In the process, we observe a necessary condition (Theorem \ref{thm:CoassObs}) for the local existence of coassociative $4$-folds.  We continue with the notation of \S\ref{ssect:G2RefTors}. \\

\indent Let $f \colon \Sigma^4 \to M^7$ denote an immersion of a coassociative $4$-fold into $M$, and let $f^*(F_{\text{G}_2}) \to \Sigma$ denote the pullback bundle.  Let $B \subset f^*(F_{\text{G}_2})$ denote the subbundle of coframes adapted to $\Sigma$, i.e., the subbundle whose fiber over $x \in \Sigma$ is
$$B|_x = \{u \in f^*(F_{\text{G}_2})|_x \colon u(T_x\Sigma) = 0 \oplus \mathsf{C}^\sharp \}$$
We recall (Proposition \ref{prop:G2ActStab}) that $\text{G}_2$ acts transitively on the set of coassociative $4$-planes with stabilizer $\text{SO}(4)$, so $B \to \Sigma$ is a well-defined $\text{SO}(4)$-bundle.  Note that on $B$, we have
$$\omega_{\mathsf{A}} = 0.$$
For the rest of \S\ref{ssect:MCCoassoc}, all of our calculations will be done on the subbundle $B \subset F_{\text{G}_2}$. \\

\indent As in \S\ref{ssect:MCAssoc}, we use the splitting $T_xM = (T_x\Sigma)^\perp \oplus T_x\Sigma \simeq \mathsf{A}^\sharp \oplus \mathsf{C}^\sharp$ to decompose $\theta$ and $\gamma$ into $\text{SO}(4)$-irreducible pieces.  The result is the identical: the connection $1$-form $\theta \in \Omega^1(B; \mathfrak{g}_2)$ takes the block form
$$\theta = \begin{bmatrix} \rho(\zeta) & -\sigma^T \\
\sigma & \zeta + \xi \end{bmatrix} = \left[ \begin{array}{c c c | c c c c}
0 & 2\zeta_3 & -2\zeta_2 & -\sigma_4 - \sigma_7 & -\sigma_3 + \sigma_8 & \sigma_2 + \sigma_5 & \sigma_1 - \sigma_6 \\
-2\zeta_3 & 0 & 2\zeta_1 & \sigma_1 & \sigma_2 & \sigma_3 & \sigma_4 \\
2\zeta_2 & -2\zeta_1 & 0  & \sigma_5 & \sigma_6 & \sigma_7 & \sigma_8 \\ \hline
\sigma_4 + \sigma_7 & -\sigma_1 & -\sigma_5 & 0 & -\zeta_1-\xi_1 & -\zeta_2 + \xi_2 & \zeta_3 + \xi_3 \\
\sigma_3 - \sigma_8 & -\sigma_2 & -\sigma_6 & \zeta_1 + \xi_1 & 0 & \zeta_3 - \xi_3 & \zeta_2 + \xi_2 \\
-\sigma_2 - \sigma_5 & -\sigma_3 & -\sigma_7 & \zeta_2 - \xi_2 & -\zeta_3 + \xi_3 & 0 & -\zeta_1 + \xi_1 \\
-\sigma_1 + \sigma_6 & -\sigma_4 & -\sigma_8 & -\zeta_3 - \xi_3 & -\zeta_2 - \xi_2 & \zeta_1 - \xi_1 & 0
\end{array}\right]\!.$$
and the $1$-form $\gamma \in \Omega^1(B; \mathbb{R}^7)$ takes the block form
$$\gamma = \begin{bmatrix}
\gamma_{\mathsf{A}} & -(\gamma_{\mathsf{C}})^T \\
\gamma_{\mathsf{C}} & (\gamma_{\mathsf{A}})_+ \\
\end{bmatrix} = \left[ \begin{array}{c c c | c c c c}
0 & \gamma_3 & -\gamma_2 & \gamma_5 & -\gamma_4 & \gamma_7 & -\gamma_6 \\
-\gamma_3 & 0 & \gamma_1 & \gamma_6 & -\gamma_7 & -\gamma_4 & \gamma_5 \\
\gamma_2 & -\gamma_1 & 0 & -\gamma_7 & -\gamma_6 & \gamma_5 & \gamma_4 \\ \hline
-\gamma_5 & -\gamma_6 & \gamma_7 & 0 & \gamma_1 & \gamma_2 & -\gamma_3 \\
\gamma_4 & \gamma_7 & \gamma_6 & -\gamma_1 & 0 & -\gamma_3 & -\gamma_2 \\
-\gamma_7 & \gamma_4 & -\gamma_5 & -\gamma_2 & \gamma_3 & 0 & \gamma_1 \\
\gamma_6 & -\gamma_5 & -\gamma_4 & \gamma_3 & \gamma_2 & -\gamma_1 & 0
\end{array}\right]\!.$$ \\
\indent In this notation, the first structure equation (\ref{eq:G2-FirstStrEqn}) on $B$ reads:
\begin{align*}
d\begin{pmatrix} 0 \\ \omega_{\mathsf{C}} \end{pmatrix} = -\left(\begin{bmatrix} \rho(\zeta) & -\sigma^T \\
\sigma & \zeta + \xi \end{bmatrix} + 2\begin{bmatrix}
\gamma_{\mathsf{A}} & -(\gamma_{\mathsf{C}})^T \\
\gamma_{\mathsf{C}} & (\gamma_{\mathsf{A}})_+ \\
\end{bmatrix} \right) \wedge \begin{pmatrix} 0 \\ \omega_{\mathsf{C}} \end{pmatrix}
\end{align*}
In particular, the first line gives
$$0 = (\sigma^T + 2(\gamma_{\mathsf{C}})^T ) \wedge \omega_{\mathsf{C}}$$
or in full detail,
\begin{equation} \label{eq:Coassoc-CondensedStrEqn}
\begin{bmatrix}
-\sigma_4 - \sigma_7 & -\sigma_3 + \sigma_8 & \sigma_2 + \sigma_5 & \sigma_1 - \sigma_6 \\
 \sigma_1 &  \sigma_2 & \sigma_3  & \sigma_4 \\
 \sigma_5 & \sigma_6 & \sigma_7  & \sigma_8 
\end{bmatrix} \wedge \begin{bmatrix} \eta^1 \\ \eta^2 \\ \eta^3 \\ \eta^4 \end{bmatrix} = -2 \begin{bmatrix}
\gamma_5 & -\gamma_4 & \gamma_7 & -\gamma_6 \\
\gamma_6 & -\gamma_7 & -\gamma_4 & \gamma_5 \\
-\gamma_7 & -\gamma_6 & \gamma_5 & \gamma_4
\end{bmatrix} \wedge \begin{bmatrix} \eta^1 \\ \eta^2 \\ \eta^3 \\ \eta^4 \end{bmatrix}
\end{equation}
Note that on $B$, the $1$-forms $\sigma_\delta$ and $\gamma_\alpha$ are semibasic, so we can write
\begin{align*}
\sigma_\delta & = S_{\delta \alpha}\, \omega^\alpha & \gamma_\beta & = T_{\beta \alpha}\, \omega^\alpha
\end{align*}
for some function $S = (S_{\delta \alpha}) \colon B \to \mathsf{V}_{1,3} \otimes \mathsf{C}$, recalling our index ranges $1 \leq p \leq 3$ and $4 \leq \alpha,\beta \leq 7$ and $1 \leq \delta \leq 8$. \\
\indent Note that the $32$ functions $S_{\delta \alpha}$ and the $16$ functions $T_{\beta \alpha}$ are not independent: the equation (\ref{eq:Coassoc-CondensedStrEqn}) shows that they satisfy $3 \binom{4}{2} = 18$ linear relations.  Explicitly:
$$\begin{bmatrix}
S_{15} - S_{24} \  \ & S_{55} - S_{64} \ \ & S_{84} + S_{45} + S_{75} - S_{34} \\
S_{16} - S_{34} \ & S_{56} - S_{74} \ & S_{54} + S_{24} + S_{46} + S_{76} \\
S_{26} - S_{35} \ & S_{66} - S_{75} \ & S_{14} - S_{64} + S_{47} + S_{77} \\
S_{17} - S_{44} \ &  S_{57} - S_{84} \ & S_{55} - S_{86} + S_{25} + S_{36} \\
S_{27} - S_{45} \ & S_{67} - S_{85} \ & -S_{65} + S_{15} + S_{37} - S_{87} \\
S_{37} - S_{46} \ & S_{77} - S_{86} \ & -S_{66} - S_{57} + S_{16} - S_{27}
\end{bmatrix} = 2\begin{bmatrix}
- T_{74} - T_{65} & - T_{64} + T_{75} & T_{44} + T_{55} \\
- T_{44} - T_{66} & T_{54} + T_{76} & -T_{74} + T_{56} \\
- T_{45} + T_{76} & T_{55} + T_{66} &  T_{64} + T_{57} \\
T_{54} - T_{67} & T_{44} + T_{77} & -T_{75} - T_{46} \\
T_{55} + T_{77} & T_{45} + T_{67} & T_{65} - T_{47} \\
T_{56} + T_{47} & T_{46} - T_{57} & T_{66} + T_{77}
\end{bmatrix}$$ \\
\indent We make two observations on this system of equations.  First, we observe the relation
\begin{align*}
(T_{44} + T_{55}) + (T_{55} + T_{66}) + (T_{55} + T_{77}) + (T_{66} + T_{77}) + (T_{44} + T_{77}) + (T_{44} + T_{66}) & = 0.
\end{align*}
Substituting (\ref{eq:G2-TorSol-2}), this equation simplifies to
$$3F + \frac{1}{24}\tau_0 = 0.$$
Substituting (\ref{eq:G2-NormalizedTors}), we have proved:

\begin{thm}\label{thm:CoassObs} 
	If a coassociative $4$-fold $\Sigma$ exists in $M$, then the following relation holds at points of $\Sigma$:
	\begin{equation}
	\tau_0 = -\frac{\sqrt{42}}{7}\,[(\tau_3)_{0,0}]^\dagger \label{eq:G2-Obs}
	\end{equation}
	In particular, if $\tau_3 = 0$ and $\tau_0$ is non-vanishing (so the torsion takes values in $(W_1 \oplus W_7 \oplus W_{14}) - (W_7 \oplus W_{14})$), then $M$ admits no coassociative $4$-folds (even locally). \\
\end{thm}

\begin{cor}\label{cor:CoassObs} 
	Fix $x \in M$. If every coassociative $4$-plane in $T_xM$ is tangent to a coassociative $4$-fold, then $\tau_0|_x = 0$ and $\tau_3|_x = 0$.
\end{cor}
\begin{proof}
	The hypotheses imply that equation (\ref{eq:G2-Obs}) holds for all coassociative $4$-planes at $x \in M$.  Thus, we have a $\text{G}_2$-invariant linear relation between $\tau_0|_x$ and $\tau_3|_x$. This implies that $\tau_0|_x = 0$ and $\tau_3|_x = 0$ by Schur's Lemma. \\
\end{proof}

\indent Second, we observe the three relations
\begin{subequations} \label{eq:Coassoc-TorsRel}
\begin{align}
(S_{17} - S_{67}) + (S_{26} + S_{56}) + (-S_{35} + S_{85}) - (S_{44} + S_{74}) & = -4(T_{45} - T_{54}) - 4(T_{67} - T_{76}) \\
S_{14} + S_{25} + S_{36} + S_{47} & = 4(T_{57} - T_{75}) - 4(T_{46} - T_{64}) \\
S_{54} + S_{65} + S_{76} + S_{87} & = 4(T_{47} - T_{74}) + 4(T_{56} - T_{65}).
\end{align}
\end{subequations}
We may now compute the mean curvature of a coassociative $4$-fold: \\
\begin{thm}\label{thm:MCCoass} 
	Let $\Sigma \subset M$ be a coassociative $4$-fold immersed in a $7$-manifold $M$ equipped with a $\text{G}_2$-structure.  Then the mean curvature vector $H$ of $\Sigma$ is given by
	$$H = -4 [(\tau_1)_{\mathsf{A}}]^\sharp + \frac{\sqrt{6}}{3} \left[(\tau_2)_{\mathsf{A}}\right]^\natural.$$
	In particular, the largest torsion class of $\text{G}_2$-structures $\varphi$ for which every coassociative $4$-fold is minimal is $W_1 \oplus W_{27}$, i.e., the class for which $d\ast\!\varphi = 0$.
\end{thm}
\begin{proof}
	Let $\beta_\alpha := \ast_{\mathsf{C}}(\omega^\alpha) \in \Omega^3(B)$ and $\text{vol}_{\mathsf{C}} = \omega^{4567}$. The mean curvature vector may be computed as follows:
	\begin{align}
	\begin{bmatrix} H^1 \\ H^2 \\ H^3 \end{bmatrix} \text{vol}_{\mathsf{C}} & = \begin{bmatrix}
	\psi_{14} & \psi_{15} & \psi_{16} & \psi_{17} \\
	\psi_{24} & \psi_{25} & \psi_{26} & \psi_{27} \\
	\psi_{34} & \psi_{35} & \psi_{36} & \psi_{37}
	\end{bmatrix} \wedge \begin{bmatrix} \beta_4 \\ \beta_5 \\ \beta_6 \\ \beta_7 
	\end{bmatrix} \notag \\
	& = \begin{bmatrix}
	-\sigma_4 - \sigma_7 & -\sigma_3 + \sigma_8 & \sigma_2 + \sigma_5 & \sigma_1 - \sigma_6 \\
	\sigma_1 & \sigma_2 & \sigma_3 & \sigma_4 \\ 
	\sigma_5 & \sigma_6 & \sigma_7 & \sigma_8
	\end{bmatrix}\! \wedge\! \begin{bmatrix} \beta_4 \\ \beta_5 \\ \beta_6 \\ \beta_7 \end{bmatrix} + \,2\!\begin{bmatrix}
	\gamma_5 & -\gamma_4 & \gamma_7 & -\gamma_6 \\
	\gamma_6 & -\gamma_7 & -\gamma_4 & \gamma_5 \\
	-\gamma_7 & -\gamma_6 & \gamma_5 & \gamma_4
	\end{bmatrix} \!\wedge\! \begin{bmatrix} \beta_4 \\ \beta_5 \\ \beta_6 \\ \beta_7 \end{bmatrix} \label{eq:Coassoc-MeanCurv}
	\end{align}
	To evaluate the first term in (\ref{eq:Coassoc-MeanCurv}), we substitute $\sigma_\delta = S_{\delta p}\omega^p$, followed by (\ref{eq:Coassoc-TorsRel}), and finally (\ref{eq:G2-TorSol-3}), to obtain:
	\begin{align*}
	\begin{bmatrix}
	-\sigma_4 - \sigma_7 & -\sigma_3 + \sigma_8 & \sigma_2 + \sigma_5 & \sigma_1 - \sigma_6 \\
	\sigma_1 & \sigma_2 & \sigma_3 & \sigma_4 \\ 
	\sigma_5 & \sigma_6 & \sigma_7 & \sigma_8
	\end{bmatrix}\! \wedge\! \begin{bmatrix} \beta_4 \\ \beta_5 \\ \beta_6 \\ \beta_7 \end{bmatrix} & = 4\begin{bmatrix} -(T_{45} - T_{54}) - (T_{67} - T_{76}) \\ (T_{57} - T_{75}) - (T_{46} - T_{64}) \\ (T_{47} - T_{74}) + (T_{56} - T_{65}) \end{bmatrix} \text{vol}_{\mathsf{C}} \\
	& = 16 \begin{bmatrix} -A_1 + C_1 \\ -A_2 + C_2 \\ -A_3 + C_3 \end{bmatrix} \text{vol}_{\mathsf{C}}
	\end{align*}
	Similarly, to evaluate the second term in (\ref{eq:Coassoc-MeanCurv}), we substitute $\gamma_\alpha = T_{\alpha p}\omega^p$ followed by (\ref{eq:G2-TorSol-3}) to obtain:
	\begin{align*}
	2\begin{bmatrix}
	\gamma_5 & -\gamma_4 & \gamma_7 & -\gamma_6 \\
	\gamma_6 & -\gamma_7 & -\gamma_4 & \gamma_5 \\
	-\gamma_7 & -\gamma_6 & \gamma_5 & \gamma_4
	\end{bmatrix} \!\wedge\! \begin{bmatrix} \beta_4 \\ \beta_5 \\ \beta_6 \\ \beta_7 \end{bmatrix} & =
	2\begin{bmatrix} -(T_{45} - T_{54}) - (T_{67} - T_{76}) \\ (T_{57} - T_{75}) - (T_{46} - T_{64}) \\ (T_{47} - T_{74}) + (T_{56} - T_{65}) \end{bmatrix} \text{vol}_{\mathsf{C}} \\
	& = 8 \begin{bmatrix} -A_1 + C_1 \\ -A_2 + C_2 \\ -A_3 + C_3 \end{bmatrix} \text{vol}_{\mathsf{C}}
	\end{align*}
	We conclude that
	$$H_p = -24 A_p + 24C_p$$
	and so (\ref{eq:G2-NormalizedTors}) yields
	$$H = -4 [(\tau_1)_{\mathsf{A}}]^\sharp + \frac{\sqrt{6}}{3} \left[(\tau_2)_{\mathsf{A}}\right]^\natural.$$
	\noindent In particular, the largest torsion class for which $H = 0$ for all coassociatives is the one for which $\tau_1 = \tau_2 = 0$, which is $W_1 \oplus W_{27}$.
\end{proof}

\section{Cayley $4$-Folds in $\text{Spin}(7)$-Structures} \label{sect:Cayley}


\subsection{Preliminaries}\label{ssect:Spin7prelim}

\indent \indent In this subsection, we define both the ambient spaces and the submanifolds of interest.  This subsection will also be used to fix notation and conventions.

\subsubsection{$\text{Spin}(7)$-Structures on Vector Spaces}\label{ssect:SpinSevVect}

\indent \indent Let $V = \mathbb{R}^8$ equipped with the standard inner product $\langle \cdot, \cdot \rangle$,  norm $\Vert \cdot \Vert$, and an orientation.  Let $\{e_1, \ldots, e_8\}$ denote the standard (orthonormal) basis of $V$, and let $\{e^1, \ldots, e^8\}$ denote the corresponding dual basis of $V^*$.  The \textit{Cayley $4$-form} is the alternating $4$-form $\Phi_0 \in \Lambda^4(V^*)$ defined by
\begin{equation*}
\Phi_0 = e^{1234} + \left(e^{12} + e^{34} \right)  \wedge \left(e^{56} + e^{78}\right) + \left( e^{13} - e^{24} \right) \wedge \left(e^{57} - e^{68}\right) +\left( -e^{14}-e^{23} \right) \wedge \left(e^{58} + e^{67}\right) + e^{5678}.
\end{equation*}

For calculations, it will be convenient to express $\Phi_0$ in the form
\begin{align*}
\Phi_0 = \tfrac{1}{24} \Phi_{ijkl} e^{ijkl},
\end{align*}
where the constants $\Phi_{ijkl} \in \left\lbrace -1, 0, 1 \right\rbrace$ are defined by this formula and skew-symmetry.

The stabilizer of $\Phi_0$ under the pullback action of $GL(V)$ is isomorphic to $\spsev,$
\begin{align}\label{eq:Sp7def}
\spsev \cong \{A \in \text{GL}(V) \colon A^*\Phi_0 = \Phi_0\}.
\end{align}
Using the isomorphism (\ref{eq:Sp7def}), we regard $\spsev$ as a subgroup of $GL(V).$ As a $\spsev$-module, $V$ is isomorphic to the spinor representation of $\spsev$, and is thus irreducible. The induced representation on $\Lambda^k (V^*)$ for $2 \leq k \leq 6$ is reducible, and there are irreducible decompositions \cite{Bry87}
\begin{align*}
\Lambda^2 (V^*) & \cong \Lambda^2_{7} \oplus \Lambda^2_{21}, \\
\Lambda^3 (V^*) & \cong \Lambda^3_{8} \oplus \Lambda^3_{48}, \\
\Lambda^4 (V^*) & \cong \Lambda^4_1 \oplus \Lambda^4_7 \oplus \Lambda^4_{27} \oplus \Lambda^4_{35}.
\end{align*}
We will need the definitions of the irreducible pieces of $\Lambda^2(V^*)$ and $\Lambda^3(V^*),$ they are
\begin{align*}
\Lambda^2_7 &= \left\lbrace \beta \in \Lambda^2 (V^*) \mid * \left( \Phi_0 \wedge \beta \right) = 3 \beta \right\rbrace, \\
\Lambda^2_{21} &= \left\lbrace \beta \in \Lambda^2 (V^*) \mid * \left( \Phi_0 \wedge \beta \right) = - \beta \right\rbrace \cong \mathfrak{spin}(7), \\
\Lambda^3_8 &= \left\lbrace * \left( \Phi_0 \wedge \alpha \right) \mid \alpha \in V^* \right\rbrace, \\
\Lambda^3_{48} &= \left\lbrace \gamma \in \Lambda^3 (V^*) \mid \Phi_0 \wedge \gamma  = 0 \right\rbrace .
\end{align*}
The irreducible decompositions of the spaces $\Lambda^5 (V^*)$ and $\Lambda^6 (V^*)$ can be obtained by applying the Hodge star operator to the decompositions of $\Lambda^3(V^*)$ and $\Lambda^2(V^*)$ respectively.

\subsubsection{$\text{Spin}(7)$-Structures on $8$-Manifolds}

\indent \indent Let $M$ be an oriented 8-manifold. A $\spsev$-structure on $M$ is a differential 4-form $\Phi \in \Omega^4 \left( M \right)$ such that at each $x \in M,$ and after identifying $T_x M$ with $V = \R^8,$ the 4-form $\Phi |_x \in \Lambda^4 \left( T^*_x M \right)$ lies in the $GL(V)$-orbit of $\Phi_0.$  That is, at each $x \in M$ there exists a coframe $u : T_x M \to \R^8$ for which $\Phi |_x = u^* \left( \Phi_0 \right).$ \\

The first order local invariants of a $\spsev$-structure are completely encoded in a $\spsev$-equivariant function
\begin{align*}
T : F_{\spsev} \to \Lambda^1 \oplus \Lambda^3_{48}
\end{align*}
called the \emph{intrinsic torsion function}, defined on the total space of the $\spsev$-frame bundle $F_{\spsev} \to M$ over $M.$

By a result of Fern\'{a}ndez \cite{Fernan86}, the intrinsic torsion function of a $\spsev$-structure is equivalent to the data of the 5-form $d \Phi.$ It follows from the decomposition of $\Lambda^5 ( V^*)$ described in the previous section that the exterior derivative of $\Phi$ takes the form
\begin{align}
d \Phi = \tau_1 \wedge \Phi + * \tau_3,
\end{align}
where $\left( \tau_1, \tau_3 \right) \in \Gamma \left( \Lambda^1 ( T^* M ) \oplus \Lambda^3_{48} ( T^* M ) \right)$. We refer to the forms $\tau_1$ and $\tau_3$ as the \emph{torsion forms} of the $\spsev$-structure.

\subsubsection{Cayley $4$-Folds}

\indent \indent Let $ \left( M, \Phi \right) $ be an 8-manifold with a $\spsev$-structure , and consider a tangent space $ \left( T_x M , \Phi |_x \right) \cong \left( V, \Phi_0 \right).$ The vector space $ \left( V, \Phi_0 \right)$ possesses a distinguished class of subspaces --- the Cayley 4-planes --- first studied by Harvey and Lawson \cite{harvey1982calibrated} in their work on calibrations.

\begin{prop}[\cite{harvey1982calibrated}]
 The Cayley 4-form $\Phi_0$ has co-mass one, meaning that
\begin{align*}
| \Phi_0 \left( x, y, z, w \right) | \leq 1
\end{align*}
for every orthonormal set $ \left\lbrace x, y, z, w \right\rbrace$ in $V \cong \R^7.$
\end{prop}

\begin{prop}[\cite{harvey1982calibrated}]
Let $K \in \text{Gr}_4 ( V )$ be a 4-plane in $V$. The following are equivalent
\begin{enumerate}[nosep]
\item If $\left\lbrace x, y, z, w \right\rbrace$ is an orthonormal basis of $K,$ then $\Phi_0 \left( x, y, z, w \right) = \pm 1$.
\item $\text{Im} \left( x \times y \times z \times w \right) = 0$ for any basis $x, y, z, w$ of $K.$ 
\item $y \times z \times w \in K$ for all $y, z, w \in K.$
\end{enumerate}
\end{prop}

The $\spsev$-action on $V$ induces $\spsev$-actions on the Grassmannians $\text{Gr}_k(V)$ of $k$-planes in $V$. These actions are transitive for $k=1,2,3,5,6,7,$ but the action is not transitive when $k=4.$ Indeed, the (proper) subset consisting of Cayley 4-planes is a $\spsev$-orbit.

\begin{prop}\label{prop:CayAlg}
The Lie group $\spsev$ acts transitively on the subset of Cayley 4-planes
\begin{align*}
\left\lbrace E \in \text{Gr}_4 ( V ) \colon \left|\Phi_0 \left( E \right)\right| = 1 \right\rbrace \subset \text{Gr}_4 ( V ) 
\end{align*}
with stabiliser isomorphic to $SU(2) \times SU(2) \times SU(2) / \left\lbrace \pm \text{Id} \right\rbrace.$
\end{prop}

\noindent \textbf{Notation:} We will denote the group $SU(2)^3 / \left\lbrace \pm \text{Id} \right\rbrace$ by $\sph.$


\subsection{Some $\text{Spin}^h(4)$-Representation Theory}\label{ssect:sphreps}

\indent \indent The Lie group $\sph$ is double-covered by the simply-connected group $\text{SU}(2)^3.$  The complex irreducible representations of $\text{SU}(2)^3$ are exactly the tensor products $\V_p \otimes \V_q \otimes \V_r$ of irreducible $\text{SU}(2)$-representations for each factor. The complex irreducible representations of $\text{SU}(2)$ are well known to be the spaces of homogeneous polynomials in two variables of fixed degree, $\V_p = \text{Sym}^p \left(\mathbb{C}\langle x,y\rangle \right)$. \\
\indent Let $\V^{\mathbb{C}}_{p,q,r}$ denote $\V_p \otimes \V_q \otimes \V_r$.  We think of $\V^{\mathbb{C}}_{p,q,r}$ as the space of homogeneous polynomials in $(u,v;x,y;w,z)$ of tridegree $(p,q,r)$.  When $p+q+r$ is even the representation $\V^{\mathbb{C}}_{p,q,r}$ descends to a representation of $\sph$, and each of these representations has a real structure induced by the map $(u,v,x,y,w,z) \mapsto (v,-u,y,-x,z,-w).$ This yields a complete description of the real representations of $\sph$. We work with real representations, letting $\V_{p,q,r}$ denote the real representation underlying  $\V^{\mathbb{C}}_{p,q,r}$. 

The Clebsch-Gordan formula applied to each $\text{SU}(2)$ representation gives the irreducible decomposition of a tensor product of $\sph$-modules:
\begin{align*}
\V_{p_1,q_1,r_1} \otimes \V_{p_2,q_2,r_2} \cong \bigoplus_{i = 0}^{|p_1-p_2|} \bigoplus_{j=0}^{|q_1 - q_2|} \bigoplus_{k=0}^{|r_1 - r_2|} \V_{p_1+p_2-2i, \,q_1+q_2-2j, \, r_1+r_2-2k}.
\end{align*}


\subsubsection{$\text{Spin}^h(4)$ as a subgroup of $\text{Spin}(7)$}

\indent \indent In our calculations we shall need a concrete realization of $\sph$ as the stabilizer of a Cayley plane. Let $\sph$ act on $V \cong \R^8$ via the identification $V \cong \V_{1,1,0} \oplus \V_{0,1,1}.$ Define a basis $\left( e_1, \dotsc, e_8 \right)$ of $V$ by
\begin{align*}
e_1 & = i \left( -ux + vy \right), & e_2 & = ux+vy, & e_3 &= -i \left( uy+vx \right), & e_4 & = -uy + vx, \\
e_5 & = i \left( xw - yz \right), & e_6 & = xw+yz, & e_7 & = -i \left( xz+yw \right), & e_8 & = -xz + yw.
\end{align*}
Then the 4-form
\begin{equation*}
e^{1234} + \left(e^{12} + e^{34} \right)  \wedge \left(e^{56} + e^{78}\right) + \left( e^{13} - e^{24} \right) \wedge \left(e^{57} - e^{68}\right) +\left( -e^{14}-e^{23} \right) \wedge \left(e^{58} + e^{67}\right) + e^{5678}
\end{equation*}
is $\sph$-invariant, and thus the action of $\sph$ on $V$ gives an embedding $\sph \subset \spsev.$ The 4-plane $\left\langle e_1, e_2, e_3, e_4 \right\rangle$ is Cayley and is preserved by the action of $\sph.$ 

\subsubsection{Decomposition of 1-forms on $V$}

\indent \indent Let $V$ be as above. We have
\begin{align*}
\Lambda^1 \left( V^* \right) = \mathsf{K} \oplus \mathsf{L},
\end{align*}
where
\begin{align*}
\K & = \left\langle e_1, e_2, e_3, e_4 \right\rangle\!, \\
\Ls & = \left\langle e_5, e_6, e_7, e_8 \right\rangle\!.
\end{align*}
As abstract $\sph$-modules, $\mathsf{K} \cong \mathsf{V}_{1,1,0}$ and $\mathsf{L} \cong \mathsf{V}_{0,1,1}.$

\subsubsection{Decomposition of 2-forms on $V$}

\indent \indent We now seek to decompose $\Lambda^2 \left( V^* \right)$ into $\sph$-irreducible submodules. As noted in \S\ref{ssect:Spin7prelim} above, $\Lambda^2 \left( V^* \right)$ splits into $\text{Spin}(7)$-irreducible submodules as
\begin{align}\label{eq:Lam2Sp7}
\Lambda^2 \left( V^* \right) \cong \Lambda^2_{7} \oplus \Lambda^2_{21}.
\end{align}
On the other hand, using $V^* \cong \K \oplus \Ls,$ we may also decompose $\Lambda^2 \left( V^* \right)$ as
\begin{align}\label{eq:Lam2Cay}
\Lambda^2 \left( V^* \right) \cong \Lambda^2_+ \left( \K \right) \oplus \Lambda^2_- \left( \K \right) \oplus \left( \K \otimes \Ls \right) \oplus  \Lambda^2_+ \left( \Ls \right) \oplus \Lambda^2_- \left( \Ls \right).
\end{align}
We will refine both decompositions (\ref{eq:Lam2Sp7}) and (\ref{eq:Lam2Cay}) into irreducible $\sph$-modules.

To begin, note first that as $\sph$-modules, we have that $\Lambda^2_- \left( \K \right) \cong \V_{2,0,0},$ $\Lambda^2_- \left( \Ls \right) \cong \V_{0,0,2},$ and $\Lambda^2_+ \left( \K \right) \cong \Lambda^2_+ \left( \Ls \right) \cong \V_{0,2,0}$ are irreducible. Thus, it remains only to decompose $\Lambda^2_7, \Lambda^2_{21},$ and $\K \otimes \Ls$.

\begin{defn}
We define
\begin{align*}
(\Lambda^2_7)_{0,2,0} & := \Lambda^2_7 \cap \left( \Lambda^2_+ \left( \K \right) \oplus \Lambda^2_+ \left( \Ls \right) \right), \\
(\Lambda^2_{7})_{1,0,1} & :=  \Lambda^2_7 \cap \left( \K \otimes \Ls \right), \\
\\
(\Lambda^2_{21})_{2,0,0} & := \Lambda^2_{21} \cap  \Lambda^2_- \left( \K \right),  \\
(\Lambda^2_{21})_{0,0,2} & := \Lambda^2_{21} \cap  \Lambda^2_- \left( \Ls \right),  \\
(\Lambda^2_{21})_{0,2,0} & := \Lambda^2_{21} \cap \left( \Lambda^2_+ \left( \K \right) \oplus \Lambda^2_+ \left( \Ls \right) \right), \\
(\Lambda^2_{21})_{1,2,1} & := \Lambda^2_{21} \cap ( \K \otimes \Ls ).
\end{align*}
\end{defn}
\noindent The reader can check that, in fact, $\Lambda^2_{-}(\mathsf{K}) \subset \Lambda^2_{21}$ and $\Lambda^2_-(\Ls) \subset \Lambda^2_{21}$, so that $(\Lambda^2_{21})_{2,0,0} = \Lambda^2_-(\mathsf{K})$ and $\left( \Lambda^2_{21}\right) _{0,0,2} = \Lambda^2_- \left( \Ls \right)$.

\begin{lem}
	The decompositions
	\begin{align*}
	\Lambda^2_7 & = (\Lambda^2_7)_{0,2,0} \oplus (\Lambda^2_7)_{1,0,1} \\
	\Lambda^2_{21} & = \Lambda^2_-(\mathsf{K}) \oplus (\Lambda^2_{21})_{0,2,0} \oplus \Lambda^2_- \left( \Ls \right) \oplus (\Lambda^2_{21})_{1,2,1}
	\end{align*}
	consist of $\sph$-irreducible submodules.
\end{lem}
Thus, the decomposition
\begin{align*}
\Lambda^2 \left( V \right) = \left[ \left( \Lambda^2_7 \right)_{0,2,0} \oplus \left( \Lambda^2_7 \right)_{1,0,1} \right] \oplus \left[  \Lambda^2_-(\mathsf{K}) \oplus (\Lambda^2_{21})_{0,2,0} \oplus \Lambda^2_- \left( \Ls \right) \oplus (\Lambda^2_{21})_{1,2,1} \right]
\end{align*}
refines (\ref{eq:Lam2Sp7}), while the decomposition
\begin{align*}
\Lambda^2 \left( V \right) = \Lambda^2_+ \left( \K \right) \oplus \Lambda^2_- \left( \K \right) \oplus \left[ (\Lambda^2_{21})_{1,2,1} \oplus (\Lambda^2_{7})_{1,0,1} \right] \oplus \Lambda^2_+ \left( \Ls \right) \oplus \Lambda^2_- \left( \Ls \right)
\end{align*}
refines (\ref{eq:Lam2Cay}).

\subsubsection{Decomposition of 3-forms on $V$}

\indent \indent We now decompose $\Lambda^3 \left( V^* \right)$ into $\sph$-irreducible submodules. As noted in \S\ref{ssect:SpinSevVect}, the $\spsev$-irreducible decomposition of $\Lambda^3 \left( V^* \right)$ is
\begin{align}\label{eq:Lam3sp7}
\Lambda^3 (V^*) & \cong \Lambda^3_{8} \oplus \Lambda^3_{48}.
\end{align}
On the other hand, using $V^* \cong \K \oplus \Ls$, we also have the decomposition
\begin{align}\label{eq:Lam3sph}
\Lambda^3 ( V^* ) \cong \Lambda^3 ( \K ) \oplus \left( \left( \Lambda^2_+ \left( \K \right) \oplus \Lambda^2_- \left( \K \right) \right) \otimes \Ls \right) \oplus \left( \K \otimes \left( \Lambda^2_+ \left( \Ls \right) \oplus \Lambda^2_- \left( \Ls \right) \right) \right) \oplus \Lambda^3 \left( \Ls \right).
\end{align}
The summands $\Lambda^3 ( \K ) \cong \K \cong \V_{1,1,0}$ and $\Lambda^3 ( \Ls ) \cong \Ls \cong \V_{0,1,1}$ are irreducible, while the others are not. We have the following decompositions as $\sph$-modules:
\begin{align*}
& \Lambda^2_+ \left( \K \right) \otimes \Ls  \cong \V_{0,3,1}  \oplus \V_{0,1,1}, &   &\K \otimes \Lambda^2_+ \left( \Ls \right) \cong \V_{1,3,0} \oplus \V_{1,1,0}, \\
& \Lambda^2_- \left( \K \right) \otimes \Ls   \cong \V_{2,1,1},  &  &\K \otimes \Lambda^2_- \left( \Ls \right)   \cong \V_{1,1,2}.
\end{align*}
We may use the $\sph$-invariant decomposition $V \cong \K \oplus \Ls$ to decompose the space $\Lambda^3_8$ into $\sph$-modules as $\Lambda^3_8 = (\Lambda^3_8)_{\K} \oplus (\Lambda^3_8)_{\Ls},$ where
\begin{align*}
(\Lambda^3_8)_{\K} := \{\ast(\alpha \wedge \Phi_0) \colon \alpha \in \K \} \cong \K, \:\:\:\:\: (\Lambda^3_8)_{\Ls} := \{\ast(\alpha \wedge \Phi_0) \colon \alpha \in \Ls \} \cong \Ls.
\end{align*}
Comparing the decompositions (\ref{eq:Lam3sp7}) and (\ref{eq:Lam3sph}) with the decompositions of their summands, it follows that the space $\Lambda^3_{48}$ decomposes into $\sph$-modules as
\begin{align}\label{eq:Lam348sph}
\Lambda^3_{48} \cong \V_{0,3,1} \oplus \V_{2,1,1} \oplus \V_{1,3,0} \oplus \V_{1,1,2} \oplus \V_{0,1,1} \oplus \V_{1,1,0}.
\end{align}

\begin{defn}
Let $\left( \Lambda^3_{48} \right)_{i,j,k}$ denote the $V_{i,j,k}$ summand in the decomposition (\ref{eq:Lam348sph}). We will also denote $\left( \Lambda^3_{48} \right)_{1,1,0}$ and $\left( \Lambda^3_{48} \right)_{0,1,1}$ by $(\Lambda^3_{48})_{\K}$ and $(\Lambda^3_{48})_{\Ls}$ respectively.
\end{defn}

\subsection{The Refined Torsion Forms}\label{ssect:CayRefTors}

\indent \indent Let $(M^8, \Phi)$ be an $8$-manifold with a $\spsev$-structure.  Fix a point $x \in M$, choose an arbitrary Cayley $4$-plane $\mathsf{K}^\sharp \subset T_xM$, and let $\mathsf{L}^\sharp \subset T_xM$ denote its orthogonal $4$-plane.  Our purpose in this section is to understand how the torsion of the $\spsev$-structure decomposes with respect to the splitting
$$T_xM = \mathsf{K}^\sharp \oplus \mathsf{L}^\sharp.$$

\subsubsection{The Refined Torsion Forms in a Local $\sph$ Frame}

\indent \indent Fix $x \in M^8$ and split $T_x^* M = \K \oplus \Ls$ as above. All calculations in this subsection a performed pointwise, and we suppress reference to $x \in M.$ By the decompositions in \S\ref{ssect:sphreps}, the torsion forms $\tau_1$ and $\tau_3$ decompose into $\sph$-irreducible pieces as follows:
\begin{subequations}\label{eq:SpSevRefTors}
\begin{align}
\tau_1 &= \left( \tau_1 \right)_\K + \left( \tau_1 \right)_{\Ls}, \\
\tau_3 &= \left( \tau_3 \right)_\K + \left( \tau_3 \right)_{\Ls} + \left( \tau_3 \right)_{0,3,1} + \left( \tau_3 \right)_{2,1,1} + \left( \tau_3 \right)_{1,3,0} + \left( \tau_3 \right)_{1,1,2},
\end{align}
\end{subequations}
where $\left( \tau_1 \right)_{*} \in \left( \Lambda^1 \right)_{*},$ and $\left( \tau_3 \right)_{*} \in \left( \Lambda^3_{48} \right)_{*}.$

We refer to $\left( \tau_1 \right)_\K, \left( \tau_1 \right)_{\Ls}, \ldots, \left( \tau_3 \right)_{1,1,2}$ as the \emph{refined torsion forms} of the $\spsev$-structure at $x$ \emph{relative to the splitting} $T_x^* M = \K \oplus \Ls.$

We next express the refined torsion forms in terms of a local $\sph$-frame. Let $\left\lbrace e_1, \ldots, e_8 \right\rbrace$ be an orthonormal basis for $T_x M$ for which $\K^{\sharp} = \text{span} \left( e_1, e_2, e_3, e_4 \right)$ and $\Ls^{\sharp} = \text{span} \left( e_5, e_6, e_7, e_8 \right).$ Let $\left\lbrace e^1, \ldots, e^8 \right\rbrace$ denote the dual basis for $T^*_x M.$ \\

\noindent \textbf{Index Ranges:} We will employ the following index ranges: $1 \leq p,q \leq 4$ and $4 \leq r, s \leq 8$ and $1 \leq a, b \leq 7$ and $1 \leq i,j,k,\ell,m \leq 8$ and $1 \leq \alpha, \beta \leq 12$. \\

\begin{defn}
Define the $2$-forms
\begin{small}
\begin{align*}
\Theta_1 & = e^{12} - e^{34}, & \Gamma_1 & = e^{12} + e^{34}, & \Omega_1 & = e^{56} - e^{78}, & \Upsilon_1 & = e^{56} + e^{78}, \\
\Theta_2 & = e^{13} + e^{24}, & \Gamma_2 & = e^{13} - e^{24}, & \Omega_2 & = e^{57} + e^{68}, & \Upsilon_2 & = e^{57} - e^{68}, \\
\Theta_3 & = e^{14} - e^{23}, & \Gamma_3 & = -e^{14} - e^{23}, & \Omega_3 & = -e^{58} + e^{67}, & \Upsilon_3 & = e^{58} + e^{67}.
\end{align*} 
\end{small}
\end{defn}

\begin{lem}
We have that:
\begin{enumerate}[label=(\alph*),nosep]
\item $\{\Theta_1, \Theta_2, \Theta_3\}$ is a basis of $(\Lambda^2_{21})_{2,0,0}$.
\item $\{ \Omega_1, \Omega_2, \Omega_3 \}$ is a basis of $(\Lambda^2_{21})_{0,0,2}$.
\item $\{\Gamma_1 - \Upsilon_1, \Gamma_2 - \Upsilon_2, \Gamma_3 - \Upsilon_3 \}$ is a basis of $(\Lambda^2_{21})_{0,2,0}$.
\item $\{\Gamma_1 - \Upsilon_1, \Gamma_2 - \Upsilon_2, \Gamma_3 - \Upsilon_3 \}$ is a basis of $(\Lambda^2_{21})_{0,2,0}$.
\item $\{\Gamma_1 + \Upsilon_1, \Gamma_2 + \Upsilon_2, \Gamma_3 + \Upsilon_3 \}$ is a basis of $(\Lambda^2_{7})_{0,2,0}$.
\end{enumerate}
\end{lem}

\begin{defn}
Define the $3$-forms
\begin{align*}
\upsilon_{i} = \iota_{e_i} \Phi_0, \\
\rho_{p} = \upsilon_p - 7 *_{\K} e_p, \:\:\: & \rho_{r} = \upsilon_{r} - 7 *_{\Ls} e_{r},
\end{align*}
and
\begin{small}
\begin{align*}
\mu_1 & = 2( e^5 \wedge \Theta_3 + e^6 \wedge \Theta_2 ), &  \nu_1 & = 2(-e^1 \wedge \Omega_3 + e^2 \wedge \Omega_2), \\
\mu_2 & = 2( -e^5 \wedge \Theta_2 + e^6 \wedge \Theta_3), & \nu_2 & = 2(e^1 \wedge \Omega_2 + e^2 \wedge \Omega_3), \\
\mu_3 & = 2 e^5 \wedge \Theta_1, & \nu_3 & = 2(e^3 \wedge \Omega_3 - e^4 \wedge \Omega_2), \\
\mu_4 & = 2 e^6 \wedge \Theta_1, & \nu_4 & = 2(-e^3 \wedge \Omega_2 - e^4 \wedge \Omega_3), \\
\mu_5 & = 2( -e^5 \wedge \Theta_3 + e^6 \wedge \Theta_2), & \nu_5 & = 2(e^3 \wedge \Omega_3 + e^4 \wedge \Omega_2), \\
\mu_6 & = 2( -e^5 \wedge \Theta_2 - e^6 \wedge \Theta_3), & \nu_6 & = 2(-e^3 \wedge \Omega_2 + e^4 \wedge \Omega_3), \\
\mu_7 & = 2(e^7 \wedge \Theta_3 - e^8 \wedge \Theta_2), & \nu_7 & = 2(e^1 \wedge \Omega_3 + e^2 \wedge \Omega_2), \\
\mu_8 & = 2(-e^7 \wedge \Theta_2 - e^8 \wedge \Theta_3), & \nu_8 & = 2(-e^1 \wedge \Omega_2 + e^2 \wedge \Omega_3), \\
\mu_9 & = 2 e^7 \wedge \Theta_1, & \nu_9 & = -2 e^1 \wedge \Omega_1, \\
\mu_{10} & = -2 e^8 \wedge \Theta_1, & \nu_{10} & = 2 e^2 \wedge \Omega_1, \\
\mu_{11} & = 2(-e^7 \wedge \Theta_3 - e^8 \wedge \Theta_2), & \nu_{11} & = 2 e^3 \wedge \Omega_1, \\
\mu_{12} & = 2(-e^7 \wedge \Theta_2 + e^8 \wedge \Theta_3), & \nu_{12} & = -2 e^4 \wedge \Omega_1, \\
 & & & \\
 \lambda_1 & = 3 \left( e^5 \wedge \Gamma_3 + e^6 \wedge \Gamma_2 \right), & \kappa_1 & = 3 \left( -e^1 \wedge \Upsilon_3 + e^2 \wedge \Upsilon_2 \right), \\
 \lambda_2 & =3 \left( -e^5 \wedge \Gamma_2 + e^6 \wedge \Gamma_3 \right), & \kappa_2 & = 3 \left( e^1 \wedge \Upsilon_2 + e^2 \wedge \Upsilon_3 \right), \\
 \lambda_3 & =  2 e^5 \wedge \Gamma_1 + e^7 \wedge \Gamma_3 - e^8 \wedge \Gamma_2, & \kappa_3 & =3 \left( e^3 \wedge \Upsilon_3 - e^4 \wedge \Upsilon_2 \right), \\
 \lambda_4 & = 2 e^6 \wedge \Gamma_1 - e^7 \wedge \Gamma_2 - e^8 \wedge \Gamma_3, & \kappa_4 & = 3 \left( -e^3 \wedge \Upsilon_2 - e^4 \wedge \Upsilon_3 \right), \\
 \lambda_5 & = -e^5 \wedge \Gamma_3 +e^6 \wedge \Gamma_2 + 2 e^7 \wedge \Gamma_1, & \kappa_5 & = -2 e^1 \wedge \Upsilon_1 + e^3 \wedge \Upsilon_3 + e^4 \wedge \Upsilon_2, \\
 \lambda_6 & = -e^5 \wedge \Gamma_2 - e^6 \wedge \Gamma_3 - 2 e^8 \wedge \Gamma_1, & \kappa_6 & = 2 e^1 \wedge \Upsilon_1 - e^3 \wedge \Upsilon_2 + e^4 \wedge \Upsilon_3, \\
 \lambda_7 & = 3 \left( -e^7 \wedge \Gamma_3 - e^8 \wedge \Gamma_2 \right), & \kappa_7 & = e^1 \wedge \Upsilon_3 + e^2 \wedge \Upsilon_2 + 2 e^3 \wedge \Upsilon_1, \\
 \lambda_8 & = 3 \left( -e^7 \wedge \Gamma_2 + e^8 \wedge \Gamma_3\right), & \kappa_8 & = -e^1 \wedge \Upsilon_2 + e^2 \wedge \Upsilon_3 - 2 e^4 \wedge \Upsilon_1. 
\end{align*}
\end{small}
\end{defn}

\begin{lem}
We have that:
\begin{enumerate}[label=(\alph*),nosep]
\item $\{ \upsilon_p \colon 1 \leq p \leq 4 \}$ is a basis of $(\Lambda^3_{8})_{1,1,0}$ and $\{ \upsilon_{p+4} \colon 1 \leq p \leq 4 \}$ is a basis of $(\Lambda^3_{8})_{0,1,1}.$
\item $\{ \rho_p \colon 1 \leq p \leq 4 \}$ is a basis of $(\Lambda^3_{48})_{1,1,0}$ and $\{ \rho_{p+4} \colon 1 \leq p \leq 4 \}$ is a basis of $(\Lambda^3_{48})_{0,1,1}.$
\item $\{\mu_{a} \colon 1 \leq a \leq 12\}$ is a basis of $(\Lambda^3_{48})_{2,1,1}$ and $\{\nu_{a} \colon 1 \leq a \leq 12\}$ is a basis of $(\Lambda^3_{48})_{1,1,2}.$
\item $\{\lambda_i \colon 1 \leq i \leq 8 \}$ is basis of $(\Lambda^3_{48})_{0,3,1}$ and $\{\kappa_i \colon 1 \leq i \leq 8 \}$ is basis of $(\Lambda^3_{48})_{1,3,0}.$
\end{enumerate}
\end{lem}

\begin{defn} \label{defn:dagger}
Let $\dagger$ denote the inverse of the isometric isomorphism
\begin{align*}
\Ls^\sharp \to \left( \Lambda^3_{48} \right)_\Ls, \:\:\: e_r \mapsto \tfrac{1}{\sqrt{42}} \rho_r.
\end{align*}
\end{defn}

We now express the refined torsion forms in terms of the above bases. That is, we define functions $A_p, B_{r}$ and $C_q, D_{s}, E_\alpha, F_\beta, X_i, Y_j$ by
\begin{subequations}\label{eq:SpSevRefTorsExpand}
\begin{align}
\left( \tau_1 \right)_{\K} & = 32 A_p e^p, & \left( \tau_3 \right)_{\K} & = 16 C_q \rho_q \\
\left( \tau_1 \right)_{\Ls} & = 32 B_r e^r, & \left( \tau_3 \right)_{\Ls} & = 16 D_s \rho_s \\
 & & \left( \tau_3 \right)_{2,1,1} & = 8 E_\alpha \mu_\alpha, \\
 & & \left( \tau_3 \right)_{1,1,2} & = 8 F_\beta \nu_\beta, \\
 & & \left( \tau_3 \right)_{0,3,1} & = 8 X_i \lambda_i, \\
 & & \left( \tau_3 \right)_{1,3,0} & = 8 Y_j \kappa_j. 
\end{align}
\end{subequations}

For future use, we record the formulae
\begin{align}\label{eq:CayIsoIso}
\left[ \left( \tau_1 \right)_\Ls \right]^\sharp & = 32 B_r e_r, & \left[ \left( \tau_3 \right)_\Ls \right]^\dagger & = 16 \sqrt{42} D_s e_s.
\end{align}

\subsubsection{The Torsion Functions $T_{ai}$}

\indent \indent Let $\left( M^8, \Phi \right)$ be an 8-manifold with a $\spsev$-structure, and let $g_{\Phi}$ denote the underlying Riemannian metric. Let $F_{\text{SO}(8)} \to M$ denote the oriented orthonormal coframe bundle of $g_{\Phi},$ and let $\omega = \left( \omega_1, \ldots, \omega_8 \right) \in \Omega^1 \left( F_{SO(8)} ; V \right)$ denote the tautological 1-form. By the Fundamental Lemma of Riemannian Geometry, there exists a unique 1-form $\psi \in \Omega^1 \left( F_{\text{SO}(8)} ; \mathfrak{so}(8) \right),$ the Levi-Civita connection form of $g_{\Phi},$ satisfying the First Structure Equation
\begin{align}\label{eq:SpSevFirstStruct}
d \omega = - \psi \wedge \omega.
\end{align}

Let $\pi : F_{\spsev} \to M$ denote the $\spsev$-coframe bundle of $M.$ Restricted to $F_{\spsev} \subset F_{
\text{SO}(8)},$ the Levi-Civita 1-form $\psi$ is no longer a connection 1-form. Indeed, according to the $\spsev$-invariant splitting $\mathfrak{so}(8) = \mathfrak{spin}(7) \oplus \R^7,$ we have the decomposition
\begin{align}\label{eq:SpSevLvCiDecomp}
\psi = \theta + 2 \gamma,
\end{align}
where $\theta = \left( \theta_{ij} \right) \in \Omega^1 \left( F_{\spsev} ; \mathfrak{spin}(7) \right)$ is a connection 1-form (the so-called \emph{natural connection} of the $\spsev$-structure $\Phi$) and $\gamma \in \Omega \left( F_{\spsev} ; \R^7 \right)$ is a $\pi$-semibasic 1-form. The inclusion $\R^7 \hookrightarrow \mathfrak{so}(8)$ is given by
\begin{small}
\begin{align*}
\left( \gamma_1, \ldots \gamma_7 \right) \mapsto \left[ \begin {array}{cccc|cccc} 0 & \gamma_{{1}} & \gamma_{{2}} & \gamma_{{
3}} & \gamma_{{4}} & \gamma_{{5}} & \gamma_{{6}} & \gamma_{{7}} \\ 
-\gamma_{{1}} & 0 & \gamma_{{3}} & -\gamma_{{2}} & \gamma_{{5}} & -\gamma_{{4}} & \gamma_{{7}} & -\gamma_{{6}} \\
-\gamma_{{2}} & -\gamma_{{3}} & 0 & \gamma_{{1}} & \gamma_{{6}} & -\gamma_{{7}} & -\gamma_{{4}} & \gamma_{{5}} \\
-\gamma_{{3}} & \gamma_{{2}} & -\gamma_{{1}} & 0 & -\gamma_{{7}} & -\gamma_{{6}} & \gamma_{{5}} & \gamma_{{4}} \\ \hline
-\gamma_{{4}} & -\gamma_{{5}} & -\gamma_{{6}} & \gamma_{{7}} & 0 & \gamma_{{1}} & \gamma_{{2}} & -\gamma_{{3}} \\
-\gamma_{{5}} & \gamma_{{4}} & \gamma_{{7}} & \gamma_{{6}} & -\gamma_{{1}} & 0 & -\gamma_{{3}} & -\gamma_{{2}} \\
-\gamma_{{6}} & -\gamma_{{7}} & \gamma_{{4}} & -\gamma_{{5}} & -\gamma_{{2}} & \gamma_{{3}} & 0 & \gamma_{{1}} \\
-\gamma_{{7}} & \gamma_{{6}} & -\gamma_{{5}} & -\gamma_{{4}} & \gamma_{{3}} & \gamma_{{2}} & -\gamma_{{1}} & 0 \end {array} \right]. 
\end{align*}
\end{small}

Since $\gamma$ is $\pi$-semibasic, we may write
\begin{align*}
\gamma_a = T_{ai} \omega_i, \:\:\: \text{for} \:\:\: 1 \leq a \leq 7
\end{align*}
for some matrix-valued function $T = \left( T_{ai} \right) : F_{\spsev} \to \text{Hom} \left( V, \R^7 \right).$ The functions $T_{ai}$ encode the torsion of the $\spsev$-structure: they are equivalent to the torsion forms $\tau  _1$ and $\tau_3$ via the isomorphism
\begin{align}\label{eq:CayGamMap}
\text{Hom} \left( V, \R^7 \right) \cong \Lambda^1_8 \oplus \Lambda^3_{48}.
\end{align}

\subsubsection{Decomposition of the Torsion Functions}

\indent \indent For our computations in \S\ref{ssect:MCCay}, we will need to express the torsion function $T_{ai}$ in terms of the functions $A_p, B_r, \ldots, Y_j.$ To this end, we continue to work on the total space of the $\spsev$-coframe bundle $\pi : F_{\spsev} \to M,$ pulling back all quantities defined on $M$ to $F_{\spsev}$. Following common convention, and similarly to \S\ref{sssect:G2DecompTors}, we omit $\pi^*$ from the notation.

The torsion forms $\tau_1$ and $\tau_3$ satisfy
\begin{align*}
d \Phi = \tau_1 \wedge \Phi + * \tau_3.
\end{align*}
Into the left-hand side we substitute $\Phi = \tfrac{1}{24} \Phi_{ijkl} \omega_{ijkl}$ and use the first structure equation (\ref{eq:SpSevFirstStruct}), together with the decomposition (\ref{eq:SpSevLvCiDecomp}), to express the $d \omega$ terms in terms of the torsion functions $T_{ai}$. Into the right-hand side we again substitute $\Phi = \tfrac{1}{24} \Phi_{ijkl} \omega_{ijkl},$ as well as the expansions (\ref{eq:SpSevRefTors}) and (\ref{eq:SpSevRefTorsExpand}).

Upon equating coefficients, we obtain a system of $56$ linear equations relating the $56$ functions $T_{ai}$ on the left-hand side to the $56$ functions $A_p, B_r, \ldots, Y_j$ on the right-hand side. One can then use a computer algebra system (we have used M\textsc{aple}) to solve this linear system for the $T_{ai}.$ We now exhibit the result, taking advantage of the $\sph$-invariant isomorphism
\begin{align*}
\text{Hom} \left( V, \R^7 \right) \cong \left( \V_{0,2,0} \oplus \V_{1,0,1} \right) \otimes \left( \K \oplus \Ls \right)
\end{align*}
to highlight the structure of the solution.

We find
\begin{small}
\begin{align*}
\left[ \begin {array}{cccc} T_{{11}} & T_{{12}} & T_{{13}} & T_{{14}} \\
T_{{21}} & T_{{22}} & T_{{23}} & T_{{24}} \\
T_{{31}} & T_{{32}} & T_{{33}} & T_{{34}}
\end {array} \right] = & \left[ \begin {array}{cccc} 2Y_{{5}} & -2Y_{{6}} & -2Y_{{7}} & 2Y_{{8}} \\
Y_{{8}}-3Y_{{2}} & -3Y_{{1}}-Y_{{7}} & 3Y_{{4}}+Y_{{6}} & 3Y_{{3}}-Y_{{5}} \\
-3Y_{{1}}+Y_{{7}} & Y_{{8}}+3Y_{{2}} & 3Y_{{3}}+Y_{{5}}& -3Y_{{4}}+Y_{{6}} \end {array} \right] \\
& + \left[ \begin {array}{cccc} -4C_{{2}}+A_{{2}} & 4C_{{1}}-A_{{1}}&-4C_{{4}}+A_{{4}} & 4C_{{3}}-A_{{3}} \\
-4C_{{3}}+A_{{3}} & 4C_{{4}}-A_{{4}} & 4C_{{1}}-A_{{1}} & -4C_{{2}}+A_{{2}} \\
-4C_{{4}}+A_{{4}} & -4C_{{3}}+A_{{3}} & 4C_{{2}}-A_{{2}}&4C_{{1}}-A_{{1}}\end {array} \right],
\end{align*}
\end{small}
corresponding to $\V_{0,2,0} \otimes \K \cong \V_{1,3,0} \oplus \V_{1,1,0},$ and
\begin{small}
\begin{align*}
\left[ \begin {array}{cccc} T_{{15}}&T_{{16}}&T_{{17}}&T_{{18}} \\
T_{{25}}&T_{{26}}&T_{{27}}&T_{{28}} \\
T_{{35}}&T_{{36}}&T_{{37}}&T_{{38}}
\end {array} \right] = & \left[ \begin {array}{cccc} -2X_{{3}}  & -2X_{{4}} & -2X_{{5}} & 2X_{{6}} \\
X_{{6}}+3X_{{2}} & -X_{{5}}-3X_{{1}} & X_{{4}}+3X_{{8}} & X_{{3}}+3X_{{7}} \\
-X_{{5}}+3X_{{1}} & -X_{{6}}+3X_{{2}} & X_{{3}}-3X_{{7}} & -X_{{4}}+3X_{{8}}
\end {array} \right] \\
& + \left[ \begin {array}{cccc} -4D_{{2}}+B_{{2}} & 4D_{{1}}-B_{{1}} & -4D_{{4}}+B_{{4}} & 4D_{{3}}-B_{{3}} \\
-4D_{{3}}+B_{{3}}&4D_{{4}}-B_{{4}} & 4D_{{1}}-B_{{1}} & -4D_{{2}}+B_{{2}} \\
4D_{{4}}-B_{{4}} & 4D_{{3}}-B_{{3}} & -4D_{{2}}+B_{{2}} & -4D_{{1}}+B_{{1}} \end {array}
 \right],
\end{align*}
\end{small}
corresponding to $\V_{0,2,0} \otimes \Ls \cong \V_{0,3,1} \oplus \V_{0,1,1},$ and
\begin{small}
\begin{align}\label{eq:CayTorsResult}
\left[ \begin {array}{cccc} T_{{41}} & T_{{42}} & T_{{43}}&T_{{44}}
\\ T_{{51}}&T_{{52}}&T_{{53}}&T_{{54}}
\\ T_{{61}}&T_{{62}}&T_{{63}}&T_{{64}}
\\ T_{{71}}&T_{{72}}&T_{{73}}&T_{{74}}
\end {array} \right] = & \left[ \begin {array}{cccc} -2E_{{12}}+E_{{4}}&2E_{{11}}+E_{{3}}& -2E_{{6}}-E_{{10}} & -2E_{{5}}+E_{{9}} \\
2E_{{11}}-E_{{3}} & 2E_{{12}}+E_{{4}} & 2E_{{5}}+E_{{9}} & -2E_{{6}}+E_{{10}} \\
2E_{{2}}-E_{{10}} & 2E_{{1}}+E_{{9}} & -2E_{{8}}-E_{{4}} & 2E_{{7}}-E_{{3}} \\
-E_{{9}}+2E_{{1}} & -2E_{{2}}-E_{{10}}&-2E_{{7}}-E_{{3}}&-2E_{{8}}+E_{{4}}
\end {array} \right] \\
& + \left[ \begin {array}{cccc} 3D_{{1}}+B_{{1}}&-3D_{{2}}-B_{{2}}&-3
D_{{3}}-B_{{3}}&3D_{{4}}+B_{{4}} \\
3D_{{2}}+B_{{2}} & 3D_{{1}}+B_{{1}} & 3D_{{4}}+B_{{4}} & 3D_{{3}}+B_{{3}} \\
3D_{{3}}+B_{{3}}&-3D_{{4}}-B_{{4}}&3D_{{1}}+B_{{1}}&-3D_{{2}}-B_{{2}}\\
3D_{{4}}+B_{{4}}&3D_{{3}}+B_{{3}}&-3D_{{2}}-B_{{2}}&-3D_{{1}}-B_{{1}}\end {array} \right], \nonumber
\end{align}
\end{small}
corresponding to $\V_{1,0,1} \otimes \K \cong \V_{2,1,1} \oplus \V_{0,1,1},$ and
\begin{small}
\begin{align*}
\left[ \begin {array}{cccc} T_{{45}}&T_{{46}}&T_{{47}}&T_{{48}}
\\ T_{{55}}&T_{{56}}&T_{{57}}&T_{{58}}
\\ T_{{65}}&T_{{66}}&T_{{67}}&T_{{68}}
\\ T_{{75}}&T_{{76}}&T_{{77}}&T_{{78}}
\end {array} \right] = & \left[ \begin {array}{cccc} 2F_{{4}}-F_{{10}}&-2F_{{3}}+F_{{9}}&2F_{{8}}+F_{{12}}&2F_{{7}}-F_{{11}} \\
-2F_{{3}}-F_{{9}} & -2F_{{4}}-F_{{10}} & -2F_{{7}}-F_{{11}} & 2F_{{8}}-F_{{12}} \\
2F_{{2}}+F_{{12}} & -2F_{{1}}-F_{{11}} & 2F_{{6}}+F_{{10}} & 2F_{{5}}-F_{{9}} \\
2F_{{1}}-F_{{11}} & 2F_{{2}}-F_{{12}} & 2F_{{5}}+F_{{9}} & F_{{10}}-2F_{{6}}
\end {array} \right] \\
& + \left[ \begin {array}{cccc} -3C_{{1}}-A_{{1}}&3C_{{2}}+A_{{2}}&3
C_{{3}}+A_{{3}}&-3C_{{4}}-A_{{4}}\\
-3C_{{2}}-A_{{2}}&-3C_{{1}}-A_{{1}}&-3C_{{4}}-A_{{4}}&-3C_{{3}}-A_{{3}} \\
-3C_{{3}}-A_{{3}} & 3C_{{4}}+A_{{4}} & -3C_{{1}}-A_{{1}} & 3C_{{2}}+A_{{2}} \\
3C_{{4}}+A_{{4}} & 3C_{{3}}+A_{{3}} & -3C_{{2}}-A_{{2}} & -3C_{{1}}-A_{{1}} \end {array}
 \right],
\end{align*}
\end{small}
corresponding to $\V_{1,0,1} \otimes \Ls \cong \V_{1,1,2} \oplus \V_{1,1,0}.$


\subsection{Mean Curvature of Cayley $4$-Folds}\label{ssect:MCCay}

\indent \indent In this subsection we derive a formula (Theorem \ref{thm:CayleyMC}) for the mean curvature of a Cayley 4-fold in an arbitrary 8-manifold $\left( M, \Phi \right)$ with $\spsev$-structure.

We continue with the notation of \S\ref{ssect:CayRefTors}, letting $\pi : F_{\spsev} \to M$ denote the $\spsev$-coframe bundle of $M$ and $\omega = \left( \omega_{\K}, \omega_{\Ls} \right) \in \Omega^1 \left( F_{\spsev} ; \K^\sharp \oplus \Ls^\sharp \right)$ denote the tautological 1-form. We remind the reader that $\theta = \left( \theta_{ij} \right) \in \Omega^1 \left( F_{\spsev} ; \mathfrak{spin}(7) \right)$ is the natural connection 1-form, and that $\gamma = \left( \gamma_{ij} \right) \in \Omega^1 \left( F_{\spsev} ; \R^7 \right)$ is a $\pi$-semibasic 1-form encoding the torsion of $\Phi.$ Here $\gamma_{ij}$ refers to the image of $\left( \gamma_1, \ldots, \gamma_7 \right)$ under the map $\R^7 \to \mathfrak{so}(8)$ defined by (\ref{eq:CayGamMap}). We have $\gamma_a = T_{ai} \omega_i$ for $T = \left( T_{ai} \right) : F_{\spsev} \to \text{Hom} \left(V, \R^7 \right).$

Let $f: \Sigma^4 \to M$ be an immersion of a Cayley 4-form into $M,$ and let $f^* \left( F_{\spsev} \right) \to \Sigma$ denote the pullback bundle. Let $B \subset f^* \left( F_{\spsev} \right)$ denote the subbundle of coframes adapted to $\Sigma,$ i.e., the subbundle whose fibre over $x \in \Sigma$ is
\begin{align*}
B|_x = \{u \in f^*(F_{\spsev})|_x \colon u(T_x\Sigma) = \K^\sharp \oplus 0 \}.
\end{align*}
We recall (Proposition \ref{prop:CayAlg}) that $\spsev$ acts transitively on the set of Cayley 4-planes, with stabilizer isomorphic to $\sph$.  It follows that $B$ is a well-defined $\sph$-bundle over $\Sigma.$ Note that on $B$ we have
\begin{align*}
\omega_{\Ls} = 0.
\end{align*}

We may exploit the splitting $T_x M = T_x \Sigma \oplus\left(  T_x \Sigma \right)^{\perp} \cong \K^\sharp \oplus \Ls^\sharp$ to decompose $\theta$ and $\gamma$ into $\sph$-irreducible pieces. To decompose the connection 1-form $\theta \in \Omega^1 \left( B ; \mathfrak{spin}(7) \right),$ we split
\begin{align*}
\mathfrak{spin}(7) \cong \Lambda^2_{21} \cong \Lambda^2_{-} \left( \K \right) \oplus \left( \Lambda^2_{21} \right)_{0,2,0} \oplus \Lambda^2_- \left( \Ls \right) \oplus \left( \Lambda^2_{21} \right)_{1,2,1},
\end{align*}
so that $\theta$ takes the block form
\begin{small}
\begin{align*}
\theta = & \begin{bmatrix} \chi+\rho_\K(\zeta) & -\sigma^T \\
\sigma & \xi+\rho_{\Ls}(\zeta) \end{bmatrix} \\
= & \left[ \begin {array}{cccc|cccc} 0 & \chi_{{1}}+\zeta_{{1}} & \chi_{{2}}+
\zeta_{{2}} & \chi_{{3}}-\zeta_{{3}} & 2\sigma_{{1}}-\sigma_{{7}} & 2
\sigma_{{2}}-\sigma_{{8}} & \sigma_{{5}}-2\sigma_{{11}} & -\sigma_{{6}}+2\sigma_{{12}} \\
-\chi_{{1}}-\zeta_{{1}} & 0 & -\chi_{{3}}-\zeta_{{3}} & \chi_{{2}}-\zeta_{{2}} & -2\sigma_{{2}}-\sigma_{{8}} & 2\sigma_{{1}}+\sigma_{{7}} & -\sigma_{{6}}-2\sigma_{{12}} & -\sigma_{{5}}-2\sigma_{{11}} \\
-\chi_{{2}}-\zeta_{{2}} & \chi_{{3}}+\zeta_{{3}} & 0 & -\chi_{{1}}+\zeta_{{1}} & -2\sigma_{{3}}-\sigma_{{5}} & -2\sigma_{{4}}-\sigma_{{6}} & -\sigma_{{7}}-2\sigma_{{9}} & \sigma_{{8}}+2\sigma_{{10}} \\ 
-\chi_{{3}}+\zeta_{{3}} & -\chi_{{2}}+\zeta_{{2}} & \chi_{{1}}-\zeta_{{1}} & 0 & 2\sigma_{{4}}-\sigma_{{6}} & -2\sigma_{{3}}+\sigma_{{5}} & \sigma_{{8}}-2\sigma_{{10}} & \sigma_{{7}}-2\sigma_{{9}} \\ \hline
-2\sigma_{{1}}+\sigma_{{7}} & 2\sigma_{{2}}+\sigma_{{8}} & 2\sigma_{{3}}+\sigma_{{5}} & -2\sigma_{{4}}+\sigma_{{6}} & 0 & -\xi_{{1}}-\zeta_{{1}} & -\xi_{{2}}-\zeta_{{2}} & \xi_{{3}}-\zeta_{{3}}\\
-2\sigma_{{2}}+\sigma_{{8}} & -2\sigma_{{1}}-\sigma_{{7}} & 2\sigma_{{4}}+\sigma_{{6}} & 2\sigma_{{3}}-\sigma_{{5}} & \xi_{{1}}+\zeta_{{1}} & 0 & -\xi_{{3}}-\zeta_{{3}} & -\xi_{{2}}+\zeta_{{2}}\\
-\sigma_{{5}}+2\sigma_{{11}} & \sigma_{{6}}+2\sigma_{{12}} & \sigma_{{7}}+2\sigma_{{9}} & -\sigma_{{8}}+2\sigma_{{10}} & \xi_{{2}}+\zeta_{{2}} & \xi_{{3}}+\zeta_{{3}} & 0 & \xi_{{1}}-\zeta_{{1}} \\
\sigma_{{6}}-2\sigma_{{12}} & \sigma_{{5}}+2\sigma_{{11}} & -\sigma_{{8}}-2\sigma_{{10}} & -\sigma_{{7}}+2\sigma_{{9}} & -\xi_{{3}}+\zeta_{{3}} & \xi_{{2}}-\zeta_{{2}} & -\xi_{{1}}+\zeta_{{1}} & 0 \end {array} \right].
\end{align*}
\end{small}
Similarly, the 1-form $\gamma \in \Omega^1 \left(B ; \R^7 \right)$ breaks into block form as
\begin{small}
\begin{align*}
\gamma = \begin{bmatrix} \phi_\K (\gamma_{0,2,0}) & -\left(\gamma_{1,0,1}\right)^T \\
\gamma_{1,0,1} & \phi_\Ls(\gamma_{0,2,0}) \end{bmatrix} = \left[ \begin {array}{cccc|cccc} 0&\gamma_{{1}}&\gamma_{{2}}&\gamma_{{
3}}&\gamma_{{4}}&\gamma_{{5}}&\gamma_{{6}}&\gamma_{{7}}
\\
-\gamma_{{1}}&0&\gamma_{{3}}&-\gamma_{{2}}&\gamma
_{{5}}&-\gamma_{{4}}&\gamma_{{7}}&-\gamma_{{6}}\\
-\gamma_{{2}}&-\gamma_{{3}}&0&\gamma_{{1}}&\gamma_{{6}}&-\gamma_{{7}}&-\gamma_{{4}}&\gamma_{{5}}\\
-\gamma_{{3}}&\gamma_{{2}}&-\gamma_{{1}}&0&-\gamma_{{7}}&-\gamma_{{6}}&\gamma_{{5}}&\gamma_{{4}}\\ \hline
-\gamma_{{4}}&-\gamma_{{5}}&-\gamma_{{6}}&\gamma_{{7}}&0&\gamma_{{1}}&\gamma_{{2}}&-\gamma_{{3}} \\
-\gamma_{{5}} & \gamma_{{4}}&\gamma_{{7}}&\gamma_{{6}}&-\gamma_{{1}}&0&-\gamma_{{3}}&-\gamma_{{2}} \\
-\gamma_{{6}}&-\gamma_{{7}}&\gamma_{{4}}&-\gamma_{{5}}&-\gamma_{{2}}&
\gamma_{{3}}&0&\gamma_{{1}} \\
-\gamma_{{7}}&\gamma_{{6}}&-\gamma_{{5}}&-\gamma_{{4}}&\gamma_{{3}}&\gamma_{{2}}&-\gamma_{{1}}&0\end {array} \right]
\end{align*}
\end{small}

In this notation, the first structure equation (\ref{eq:SpSevFirstStruct}) on $B$ reads
\begin{align}
d \left( \begin{array}{c} \omega_\K \\
0 
\end{array} \right) = - \left( \begin{bmatrix} \chi+\rho_\K(\zeta) & -\sigma^T \\
\sigma & \xi+\rho_{\Ls}(\zeta) \end{bmatrix} + 2 \begin{bmatrix} \phi_\K (\gamma_{0,2,0}) & -\left(\gamma_{1,0,1}\right)^T \\
\gamma_{1,0,1} & \phi_\Ls(\gamma_{0,2,0}) \end{bmatrix} \right) \wedge \left( \begin{array}{c} \omega_\K \\
0 
\end{array} \right).
\end{align}
In particular, the second line gives
\begin{align}
\left( \sigma +2 \gamma_{1,0,1} \right) \wedge \omega_{\K} = 0,
\end{align}
or in detail
\begin{small}
\begin{align}\label{eq:CaydomLcons}
& \left[ \begin{array}{cccc}
-2\sigma_{{1}}+\sigma_{{7}} & 2\sigma_{{2}}+\sigma_{{8}} & 2\sigma_{{3}}+\sigma_{{5}} & -2\sigma_{{4}}+\sigma_{{6}} \\
-2\sigma_{{2}}+\sigma_{{8}} & -2\sigma_{{1}}-\sigma_{{7}} & 2\sigma_{{4}}+\sigma_{{6}} & 2\sigma_{{3}}-\sigma_{{5}} \\
-\sigma_{{5}}+2\sigma_{{11}} & \sigma_{{6}}+2\sigma_{{12}} & \sigma_{{7}}+2\sigma_{{9}} & -\sigma_{{8}}+2\sigma_{{10}}  \\
\sigma_{{6}}-2\sigma_{{12}} & \sigma_{{5}}+2\sigma_{{11}} & -\sigma_{{8}}-2\sigma_{{10}} & -\sigma_{{7}}+2\sigma_{{9}} \end {array} \right] \wedge \left[ \begin{array}{c}
\omega_1 \\
\omega_2 \\
\omega_3 \\
\omega_4
\end{array} \right] \\
& = -2  \left[ \begin{array}{cccc} -\gamma_{{4}} & -\gamma_{{5}} & -\gamma_{{6}} & \gamma_{{7}} \\
-\gamma_{{5}} & \gamma_{{4}} & \gamma_{{7}} & \gamma_{{6}} \\
-\gamma_{{6}} & -\gamma_{{7}} & \gamma_{{4}} & -\gamma_{{5}} \\
-\gamma_{{7}} & \gamma_{{6}} & -\gamma_{{5}} & -\gamma_{{4}} \end {array} \right] \wedge \left[ \begin{array}{c}
\omega_1 \\
\omega_2 \\
\omega_3 \\
\omega_4
\end{array} \right]. \nonumber
\end{align}
\end{small}

Note that on $B$ the 1-forms $\sigma$ and $\gamma$ are semibasic. We have
\begin{align*}
\sigma_{\alpha} & = S_{\alpha p} \omega_p, & \gamma_{r} = T_{r q} \omega_q,
\end{align*}
for some function $S = \left( S_{\alpha p} \right) : B \to \V_{1,2,1} \otimes \K,$ recalling our index ranges $1 \leq \alpha \leq 12$ and $1 \leq p, q \leq 4$ and $4 \leq r \leq 8.$

Now, the 48 functions $S_{\alpha p}$ and the 16 functions $T_{r q}$ are not independent: equation (\ref{eq:CaydomLcons}) amounts to 24 linear relations among them, 8 of which involve only the functions $S_{\alpha p}.$ In our calculation of the mean curvature vector of a Cayley 4-fold we shall require only the following 4 of the relations:
\begin{small}
\begin{align}\label{eq:CayStoT}
\left[ \begin{array}{c}
S_{{8\,2}}-2S_{{3\,3}}+2S_{{4\,4}}-S_{{5\,3}}+2S_{{1\,1}}-2S_{{2\,2
}}-S_{{6\,4}}-S_{{7\,1}} \\
-S_{{8\,1}}-2S_{{3\,4}}-2S_{{4\,3}}+S_{{5\,4}}+2S_{{1\,2}}+2S_{{2\,1
}}-S_{{6\,3}}+S_{{7\,2}} \\
S_{{8\,4}}-2S_{{11\,1}}-2S_{{12\,2}}+S_{{5\,1}}-2S_{{9\,3}}-2S_{{10\,4}}-S_{{6\,2}}-S_{{7\,3}} \\
S_{{8\,3}}-2S_{{11\,2}}+2S_{{12\,1}}-S_{{5\,2}}-2S_{{9\,4}}+2S_{{10\,3}}-S_{{6\,1}}+S_{{7\,4}}
\end{array} \right] =6 \left[ \begin{array}{c}
T_{{74}}+T_{{41}}+T_{{52}}+T_{{
63}} \\
-T_{{73}}-T_{{42}}+T_{{51}}-T_{{
64}} \\
T_{{72}}-T_{{43}}+T_{{54}}+T_{{61}} \\
T_{{71}}+T_{{44}}+T_{{53}}-T_{{6\,2}} \end{array} \right]\!.
\end{align}
\end{small}

\begin{thm}\label{thm:CayleyMC}
Let $\Sigma \subset M$ be a Cayley 4-fold immersed in an 8-manifold $M$ with $\spsev$-structure. Then the mean curvature vector $H$ of $\Sigma$ is given by
\begin{align*}
H = - \left[ \left( \tau_1 \right)_\Ls \right]^\sharp - \tfrac{\sqrt{42}}{7} \left[ \left( \tau_3 \right)_\Ls \right]^\dagger.
\end{align*}
In particular, the largest torsion class of $\spsev$-structures for which every Cayley 4-fold is minimal is the class of torsion-free $\spsev$-structures.
\end{thm}
 \begin{proof}
 Let $\beta_p = *_{\K} \omega_p \in \Omega^3 \left( B \right),$ and let $\text{vol}_\K = \omega_{1234}.$ The components of the mean curvature vector $H$ of $\Sigma$ may be computed as follows:
\begin{small}
 \begin{align}
 \begin{bmatrix} H_5 \\ H_6 \\ H_7 \\ H_8 \end{bmatrix} \text{vol}_{\K} & = \begin{bmatrix}
\psi_{51} & \psi_{52} & \psi_{53} & \psi_{54} \\
\psi_{61} & \psi_{62} & \psi_{63} & \psi_{64} \\
\psi_{71} & \psi_{72} & \psi_{73} & \psi_{74} \\ 
\psi_{81} & \psi_{82} & \psi_{83} & \psi_{84} 
\end{bmatrix} \wedge  \begin{bmatrix}
\beta_1 \\ \beta_2 \\ \beta_3 \\ \beta_4
\end{bmatrix} \nonumber \\
& = \left[ \begin{array}{cccc}
-2\sigma_{{1}}+\sigma_{{7}} & 2\sigma_{{2}}+\sigma_{{8}} & 2\sigma_{{3}}+\sigma_{{5}} & -2\sigma_{{4}}+\sigma_{{6}} \\
-2\sigma_{{2}}+\sigma_{{8}} & -2\sigma_{{1}}-\sigma_{{7}} & 2\sigma_{{4}}+\sigma_{{6}} & 2\sigma_{{3}}-\sigma_{{5}} \\
-\sigma_{{5}}+2\sigma_{{11}} & \sigma_{{6}}+2\sigma_{{12}} & \sigma_{{7}}+2\sigma_{{9}} & -\sigma_{{8}}+2\sigma_{{10}}  \\
\sigma_{{6}}-2\sigma_{{12}} & \sigma_{{5}}+2\sigma_{{11}} & -\sigma_{{8}}-2\sigma_{{10}} & -\sigma_{{7}}+2\sigma_{{9}} \end {array} \right] \wedge  \begin{bmatrix}
\beta_1 \\ \beta_2 \\ \beta_3 \\ \beta_4
\end{bmatrix} \label{eq:CayMCCalc} \\
& + 2  \left[ \begin{array}{cccc} -\gamma_{{4}} & -\gamma_{{5}} & -\gamma_{{6}} & \gamma_{{7}} \\
-\gamma_{{5}} & \gamma_{{4}} & \gamma_{{7}} & \gamma_{{6}} \\
-\gamma_{{6}} & -\gamma_{{7}} & \gamma_{{4}} & -\gamma_{{5}} \\
-\gamma_{{7}} & \gamma_{{6}} & -\gamma_{{5}} & -\gamma_{{4}} \end {array} \right] \wedge  \begin{bmatrix}
\beta_1 \\ \beta_2 \\ \beta_3 \\ \beta_4
\end{bmatrix}. \nonumber
 \end{align}
 \end{small}
 To evaluate the first term in (\ref{eq:CayMCCalc}), we substitute $\sigma_{\alpha}= S_{\alpha p} \omega_p,$ followed by (\ref{eq:CayStoT}), and finally (\ref{eq:CayTorsResult}), to obtain
\begin{small}
\begin{align*}
\left[ \begin{array}{c}
-2S_{{1\,1}}+S_{{7,1}}+2S_{{2\,2}}+S_{{8\,2}}+2S_{{3\,3}}+S_{{5\,3}}-
2S_{{4\,4}}+S_{{6\,4}} \\
-2S_{{2\,1}}+S_{{8\,1}}-2S_{{1\,2}}-S_{{7\,2}}+2S_{{4\,3}}+S_{{6\,3}}+
2S_{{3\,4}}-S_{{5\,4}} \\
-S_{{5\,1}}+2S_{{11\,1}}+S_{{6\,2}}+2S_{{12\,2}}+S_{{7\,3}}+2S_{{9\,3}
}-S_{{8\,4}}+2S_{{10\,4}} \\
S_{{6\,1}}-2S_{{12\,1}}+S_{{5\,2}}+2S_{{11\,2}}-S_{{8\,3}}-2S_{{10\,3}
}-S_{{7\,4}}+2S_{{9\,4}}
\end{array} \right] \text{vol}_\K = \left[ \begin{array}{c}
-24 B_5 - 72 D_5 \\
-24 B_6 - 72 D_6 \\
-24 B_7 - 72 D_7 \\
-24 B_8 - 72 D_8
\end{array} \right] \text{vol}_\K.
\end{align*}
\end{small}
Similarly, to evaluate the second term in (\ref{eq:CayMCCalc}), we substitute $\gamma_r = T_{rq} \omega_q,$ followed by (\ref{eq:CayTorsResult}), to obtain:
\begin{align*}
2 \left[ \begin{array}{c}
-T_{{41}}-T_{{52}}-T_{{63}}+T_{{74}} \\
-T_{{51}}+T_{{42}}+T_{{73}}+T_{{64}} \\
-T_{{61}}-T_{{72}}+T_{{43}}-T_{{54}} \\
-T_{{71}}+T_{{62}}-T_{{53}}-T_{{44}}
\end{array} \right] \text{vol}_\K = \left[ \begin{array}{c}
-8 B_5 - 24 D_5 \\
-8 B_6 - 24 D_6 \\
-8 B_7 - 24 D_7 \\
-8 B_8 - 24 D_8
\end{array} \right] \text{vol}_\K.
\end{align*}
We conclude that $H_{r} = -32 B_r - 96 D_r,$ and so (\ref{eq:CayIsoIso}) yields that
\begin{align*}
H = - \left[ \left( \tau_1 \right)_\Ls \right]^\sharp - \tfrac{\sqrt{42}}{7} \left[ \left( \tau_3 \right)_\Ls \right]^\dagger
\end{align*}
\end{proof}
 
 \subsubsection{The Second Fundamental Form of a Cayley 4-fold}
 
\indent \indent The mean curvature vector is not the only part of the second fundamental form of a Cayley 4-fold that can be written in terms of the torsion forms of the ambient $\spsev$-structure. For a Cayley 4-fold, the second fundamental form is naturally a section of a vector bundle modeled on the $\sph$-representation
 \begin{align}\label{eq:CayIIDecomp}
 \text{Sym}^2 \left( \K \right) \otimes \Ls & \cong \left( \V_{2,2,0} \oplus \R \right) \otimes \Ls \\
  & \cong \V_{2,3,1} \oplus \V_{0,3,1} \oplus \V_{2,1,1} \oplus \Ls, \nonumber
 \end{align}
 and so the second fundamental form naturally decomposes into four pieces, with the piece corresponding to the $\Ls$ summand equal to the mean curvature vector.
 
 By calculations similar to those performed in the proof of Theorem \ref{thm:CayleyMC}, it is possible to show that the piece of the second fundamental form corresponding to the $\V_{0,3,1}$ summand in (\ref{eq:CayIIDecomp}) is identically zero, while the piece corresponding to the $\V_{2,1,1}$ summand in (\ref{eq:CayIIDecomp}) is proportional to the refined torsion form $\left( \tau_3 \right)_{2,1,1}.$

\pagebreak

\bibliographystyle{plain}
\bibliography{MeanCurvRef}

\Addresses

\end{document}